\documentclass[11pt]{amsart}

\usepackage[
paper=a4paper,
text={147mm,230mm},centering
]{geometry}

\iftrue
\makeatletter
\def\@settitle{%
  \vspace*{-20pt}
  \begin{flushleft}%
    \baselineskip14\p@\relax
    \normalfont\bfseries\LARGE
    \@title
  \end{flushleft}%
}
\def\@setauthors{%
  \begingroup
  \def\thanks{\protect\thanks@warning}%
  \trivlist
  \large \@topsep30\p@\relax
  \advance\@topsep by -\baselineskip
  \item\relax
  \author@andify\authors
  \def\\{\protect\linebreak}%
  \authors
  \ifx\@empty\contribs
  \else
    ,\penalty-3 \space \@setcontribs
    \@closetoccontribs
  \fi
  \normalfont
  \@setaddresses
  \endtrivlist
  \endgroup
}
\def\@setaddresses{\par
  \nobreak \begingroup\raggedright
  \small
  \def\author##1{\nobreak\addvspace\smallskipamount}%
  \def\\{\unskip, \ignorespaces}%
  \interlinepenalty\@M
  \def\address##1##2{\begingroup
    \par\addvspace\bigskipamount\noindent
    \@ifnotempty{##1}{(\ignorespaces##1\unskip) }%
    {\ignorespaces##2}\par\endgroup}%
  \def\curraddr##1##2{\begingroup
    \@ifnotempty{##2}{\nobreak\noindent\curraddrname
      \@ifnotempty{##1}{, \ignorespaces##1\unskip}\/:\space
      ##2\par}\endgroup}%
  \def\email##1##2{\begingroup
    \@ifnotempty{##2}{\smallskip\nobreak\noindent E-mail address%
      \@ifnotempty{##1}{, \ignorespaces##1\unskip}\/:\space
      \ttfamily##2\par}\endgroup}%
  \def\urladdr##1##2{\begingroup
    \def~{\char`\~}%
    \@ifnotempty{##2}{\nobreak\noindent\urladdrname
      \@ifnotempty{##1}{, \ignorespaces##1\unskip}\/:\space
      \ttfamily##2\par}\endgroup}%
  \addresses
  \endgroup
  \global\let\addresses=\@empty
}
\def\@setabstracta{%
    \ifvoid\abstractbox
  \else
    \skip@25\p@ \advance\skip@-\lastskip
    \advance\skip@-\baselineskip \vskip\skip@
    \box\abstractbox
    \prevdepth\z@ % because \abstractbox is a vtop
    \vskip-10pt
  \fi
}
\renewenvironment{abstract}{%
  \ifx\maketitle\relax
    \ClassWarning{\@classname}{Abstract should precede
      \protect\maketitle\space in AMS document classes; reported}%
  \fi
  \global\setbox\abstractbox=\vtop \bgroup
    \normalfont\small
    \list{}{\labelwidth\z@
      \leftmargin0pc \rightmargin\leftmargin
      \listparindent\normalparindent \itemindent\z@
      \parsep\z@ \@plus\p@
      
    }%
    \item[\hskip\labelsep\bfseries\abstractname.]%
}{%
  \endlist\egroup
  \ifx\@setabstract\relax \@setabstracta \fi
}
\def\section{\@startsection{section}{1}%
  \z@{-1.2\linespacing\@plus-.5\linespacing}{.8\linespacing}%
  {\normalfont\bfseries\large}}
\def\subsection{\@startsection{subsection}{2}%
  \z@{-.8\linespacing\@plus-.3\linespacing}{.3\linespacing\@plus.2\linespacing}%
  {\normalfont\bfseries}}
\def\subsubsection{\@startsection{subsubsection}{3}%
  \z@{.7\linespacing\@plus.1\linespacing}{-1.5ex}%
  {\normalfont\itshape}}
\def\@secnumfont{\bfseries}
\makeatother
\fi %\iftrue/false for amsart.cls hack

\usepackage{amssymb}
\usepackage[all]{xy}
\usepackage{graphicx,color,float}
\usepackage[bookmarks]{hyperref}
\usepackage{pinlabel}
\usepackage[textwidth=1.3in,color=yellow]{todonotes}
\usepackage{amsmath}
\usepackage{pb-diagram,pb-xy}
\textwidth=\paperwidth \advance\textwidth by-2\oddsidemargin
\advance\oddsidemargin by-1in
\evensidemargin=\oddsidemargin
\calclayout

\def\H{\mathcal{H}}
\def\K{\mathcal{K}}

\def\Z{\mathbb{Z}}
\def\Q{\mathbb{Q}}
\def\R{\mathbb{R}}
\def\C{\mathbb{C}}
\def\d{\partial}

\def\tilde{\widetilde}
\def\Cov{\operatorname{Cov}}

\def\rank{\operatorname{rank}}

\def\Ker{\operatorname{Ker}}
\def\Coker{\operatorname{Coker}}
\def\Im{\operatorname{Im}}
\def\Hom{\operatorname{Hom}}

\def\GL{\operatorname{GL}}

\def\Ext{\operatorname{Ext}}

\def\inc{\operatorname{inc}}

\def\+{\oplus}
\theoremstyle{plain}
\newtheorem{theorem}{Theorem}[section]
\newtheorem{proposition}[theorem]{Proposition}
\newtheorem*{corollaryD}{Corollary D}
\newtheorem{lemma}[theorem]{Lemma}
\newtheorem*{theoremA}{Theorem A}
\newtheorem*{theoremB}{Theorem B}
\newtheorem*{theoremC}{Theorem C}
\newtheorem*{claim}{Claim}
\newtheorem*{claim1}{Claim 1}
\newtheorem*{claim2}{Claim 2}
\theoremstyle{definition}
\newtheorem{definition}[theorem]{Definition}

\newtheorem*{example0}{Example}

\newtheorem{remark}[theorem]{Remark}
\newtheorem*{remark0}{Remark}
\newtheorem*{convention}{Convention}
\def\to{\mathchoice{\longrightarrow}{\rightarrow}{\rightarrow}{\rightarrow}}
\makeatletter
\newcommand{\shortxra}[2][]{\ext@arrow 0359\rightarrowfill@{#1}{#2}}
\def\longrightarrowfill@{\arrowfill@\relbar\relbar\longrightarrow}
\newcommand{\longxra}[2][]{\ext@arrow 0359\longrightarrowfill@{#1}{#2}}
\renewcommand{\xrightarrow}[2][]{\mathchoice{\longxra[#1]{#2}}%
  {\shortxra[#1]{#2}}{\shortxra[#1]{#2}}{\shortxra[#1]{#2}}}
\makeatother
\begin{document}

\title [Whitney towers and Link invariants]%
{Whitney towers, Gropes %Blanchfield forms,
and Casson-Gordon style invariants of links}

\author{Min Hoon Kim}

\address{Department of Mathematics\\
Pohang University of Science and Technology\\
Pohang, Gyungbuk 790--784\\
Republic of Korea}

\email{kminhoon@gmail.com}

% Instead of \subjclass[2000]{...}, 
% I put a workaround for the old versions of amsart.
%\def\subjclassname{\textup{2000} Mathematics Subject Classification}
%\expandafter\let\csname subjclassname@1991\endcsname=\subjclassname
%\expandafter\let\csname subjclassname@2000\endcsname=\subjclassname
%\subjclass{Primary 20J05, 57M07, Secondary 55P60, 57M27}

%\keywords{Friedl-Powell invariant, Link concordance, Whitney tower, Casson-Gordon invariant}

\begin{abstract}
In this paper, we prove a conjecture of Friedl and Powell that their Casson-Gordon type invariant of 2-component link with linking number one is actually an obstruction to being height 3.5 Whitney tower/grope concordant to the Hopf Link. The proof employs the notion of solvable cobordism of 3-manifolds with boundary, which was introduced by Cha. We also prove that the Blanchfield form and the Alexander polynomial of links in $S^3$ give obstructions to height 3 Whitney tower/grope concordance. This generalizes the results of Hillman and Kawauchi.
\end{abstract}

\maketitle

\section{Introduction}
\label{section:introduction}

In the study of topological knot concordance, various invariants were introduced in seminal papers including \cite{Levine:1969-1}, \cite{Casson-Gordon:1978-1,Casson-Gordon:1986-1}, and \cite{Cochran-Orr-Teichner:2003-1}. All of these invariants can be extracted from the 0-surgery manifolds of knots. Influenced by these works, the link slicing problem has been studied extensively using various covers of the 0-surgery manifolds of links. For example, \cite{Harvey:2006-1}, \cite{Cochran-Harvey-Leidy:2008-1}, and \cite{Horn:2010-1} used Cheeger-Gromov $\rho$-invariants from PTFA (poly-torsion-free-abelian) covers. In \cite{Cha:2010-1} and \cite{Cha:2009-1}, Hirzebruch type invariants from iterated prime power fold covers are defined and used.

In general link concordance problems, it is known that zero surgery manifolds do not reveal full information. For example, for 2-component links with linking number one, aformentioned invariants automatically vanish. In fact, those invariants are obtained from solvable covers of zero surgery manifolds. For 2-component links with linking number one, there are no non-trivial solvable covers of the zero surgery manifolds (and consequently aformentioned invariants vanish) because they have perfect fundamental groups. Also, there is an in-depth study which presents related results about link concordance versus zero surgery homology cobordism, see \cite{Cha-Powell:2013-1}.

Recently, for 2-component links with linking number one, S.\ Friedl and M.\ Powell \cite{Friedl-Powell:2011-1} introduced a Casson-Gordon style metabelian invariant $\tau(L,\chi)$ by considering another closed 3-manifold obtained from the link exterior.  Also, they found new 2-component links with linking number one which are not concordant to the Hopf link. The aim of this paper is to give a better understanding of $\tau(L,\chi)$ in the context of symmetric Whitney towers and gropes in dimension 4.

\subsection*{Friedl-Powell invariant $\tau(L,\chi)$}
To describe our main result, we briefly summarize the construction and main result in \cite{Friedl-Powell:2011-1}. (For more details, see Section~\ref{existenceofmetabolizer}.)

Let $L$ be an ordered,  oriented 2-component link with linking number 1 in $S^3$ and $H$ be the Hopf link. Define $M_L$ to be the closed 3-manifold obtained by gluing the exteriors of $L$ and $H$ along their boundary, identifying the meridians of corresponding components.  For a prime $p$, choose a homomorphism $\varphi\colon  H_1(M_L;\Z)\to  \Z/p^i\times \Z/p^j$ which sends two meridians of $L$ to the standard basis $(1,0)$ and $(0,1)$, respectively. Let $M_L^\varphi\to M_L^{\vphantom{\phi}}
$ be the $p^{i+j}$-fold covering induced by~$\varphi$.  For a prime $q$ and a character $\chi\colon H_1(M_L^\varphi;\Z)\to \Z/q^k$, Friedl and Powell define an invariant
\[\tau(L,\chi)\in  L^0(\mathbb{C}(\H))\mathbin{\mathop{\otimes}_{\Z}} \Z[1/q]\]
in \cite[Section 3.2]{Friedl-Powell:2011-1} (see also our Definition \ref{definition:Friedl-Powell-invariant}). Here, $\H=\Z^3$, $\C(\H)$ is the quotient field of the group ring $\C[\H]$, and $L^0(\C(\H))$ is the Witt group of finite dimensional non-singular sesquilinear forms over~$\C(\H)$.  The main result of \cite{Friedl-Powell:2011-1} essentially says that if $L$ is concordant to $H$, then  $\tau(L,\chi)$ vanishes \cite[Theorem 3.5]{Friedl-Powell:2011-1}. For a precise definition of the vanishing of $\tau(L,\chi)$, see  Definition \ref{definition:vanishing-Friedl-Powell-invariant}. We omit the precise statement here because we need to discuss some technicality including the choice of a metabolizer of the linking form.

\subsection*{Symmetric Whitney tower/grope concordance and $\tau(L,\chi)$}
The symmetric Whitney towers and gropes are approximations of embedded surfaces which play the central role in the study of topological 4-manifolds. For example, a special kind of grope with caps gives a topologically embedded disk in the disk embedding theorem of \cite{Freedman-Quinn:1990-1}. Also,  using symmetric Whitney towers and gropes, T.\ Cochran, K.\ Orr, and P.\ Teichner developed the filtration theory of the knot concordance group \cite{Cochran-Orr-Teichner:2003-1}.  It turns out that the structure of this filtration theory is extremely rich (for example, see \cite{Cochran-Orr-Teichner:2004-1}, \cite{Cochran-Teichner:2003-1}, \cite{Cochran-Harvey-Leidy:2009-01}, \cite{Cochran-Harvey-Leidy:2009-02}, and \cite{Cha:2010-01}). For links, we are mainly interested in two equivalence relations,  \emph{height $h$ Whitney tower concordance} and \emph{height $h$ grope concordance}.  (For precise definitions, see \cite[Definitions 2.12, 2.15]{Cha:2012-1}.)

We remark that J.\ Conant, R.\ Schneiderman, and Teichner developed another interesting filtration theory using coarser notion called order $n$ Whitney tower concordance (for survey and references, we refer \cite{Conant-Schneiderman-Teichner:2011-1}). 
It is not our purpose to study this asymmetric filtration theory of Conant-Schneiderman-Teichner. We focus on the finer equivalence relations,  symmetric Whitney tower/grope concordance.

Our main result, Theorem A, says that the Friedl-Powell invariant $\tau(L,\chi)$ can be understood in terms of symmetric Whitney tower/grope concordance as conjectured in \cite[Remark~1.3.(5)]{Friedl-Powell:2011-1}:
\begin{theoremA}Suppose that $L$ is a 2-component link  with linking number 1 and $H$ is the Hopf link. If $L$ and $H$ are height 3.5 Whitney tower \textup{(}or grope\textup{)} concordant, then the Friedl-Powell invariant $\tau(L,\chi)$ vanishes for $L$ in the sense of Definition \ref{definition:vanishing-Friedl-Powell-invariant}.
\end{theoremA}
 In the proof, we use the notion of \emph{$h$-solvable cobordism}, introduced by J.\ C.\ Cha in \cite{Cha:2012-1} (for the definition, see Section~\ref{subsection:definition-solvable-cobordism}). By \cite[Theorem~2.13]{Cha:2012-1}, if two links $L$ and $L'$ are height $(h+2)$ Whitney tower/grope concordant, then their exteriors $X_L$ and $X_{L'}$ are $h$-solvable cobordant for all $h\in \frac{1}{2}\mathbb{Z}_{\geq 0}$. Actually, we prove Theorem A in Section \ref{subsection:proof-of-main-theorem} under the (potentially) weaker assumption that there exists a 1.5-solvable cobordism between the exteriors $X_L$ and $X_H$.

\begin{remark0}~
\begin{enumerate}
\item In \cite[Theorem 9.11]{Cochran-Orr-Teichner:2003-1}, Cochran, Orr, and Teichner proved that if a knot $K$ bounds a Whitney tower/grope of height 3.5 in $D^4$, or more generally if $K$ is 1.5-solvable, then the Casson-Gordon invariant $\tau(K,\chi)$ vanishes. Our result can be viewed as an analogue for 2-component links with linking number~1.
\item Theorem A is strictly stronger than \cite[Theorem 3.5]{Friedl-Powell:2011-1} by the following known fact: for any integer $n>2$, there are links which are height $n$ grope concordant to $H$ but not height $n.5$ Whitney tower concordant to $H$ (in particular, not concordant to $H$) \cite[Theorem 4.1]{Cha:2012-1}.  

\end{enumerate}
\end{remark0} 
\subsection*{Symmetric Whitney tower/grope concordance and abelian invariants}

In \cite[Theorem~1.1]{Cochran-Orr-Teichner:2003-1}, Cochran, Orr, and Teichner proved that a Seifert form of a knot $K$ is metabolic if and only if $K$ bounds a height 2.5 grope in $D^4$. By \cite[Corollary~2]{Schneiderman:2006-1} and \cite[Theorem~8.12]{Cochran-Orr-Teichner:2003-1}, this condition is equivalent to that $K$ bounds a height 2.5 Whitney tower in $D^4$.  Motivated from this result, in Section \ref{abelianinvariants}, we prove that Blanchfield form and the multivariable Alexander polynomial are actually obstructions to height 3 Whitney tower/grope concordance.

Abelian link concordance invariants are studied by A. Kawauchi \cite{Kawauchi:1978-1} and J. Hillman \cite{Hillman:2012-1}. To state our main result, we recall their notations (for details, see~Section~\ref{abelianinvariants}) and main results. Let $L$ be a $\mu$-component link and $X_L$ be the exterior of $L$. Denote $\Z[t_1^\pm,\ldots,t_\mu^\pm]$ by $\Lambda_\mu$.  The ring $\Lambda_\mu$ is endowed with the involution $-\colon t_i\mapsto t_i^{-1}$. Let $S$ be the multiplicative set generated by $\{t_1-1,\ldots,t_\mu-1\}$. Denote by $\Lambda_{\mu S}$ the localization of $\Lambda_\mu$ with respect to~$S$. Let $\K$ be the quotient field of~$\Lambda_\mu$. Using the Hurewicz map $\pi_1(X_L)\to \Z^\mu$, we define $H_*(X_L;\Lambda_\mu)$ and $H_*(X_L;\Lambda_{\mu S})$.

In \cite[Chapter~2]{Hillman:2012-1}, Hillman defined $\K/\Lambda_{\mu S}$-valued the localized Blanchfield form $b_L$ defined on the quotient of the torsion submodule of $H_1(X_L;\Lambda_{\mu S})$ by its maximal pseudonull-submodule. Also, he proved that the Witt-class of $b_L$, denoted by $[b_L]$, in the Witt group $W(\K, \Lambda_{\mu S},-)$ is a link concordance invariant.
 
In \cite{Kawauchi:1978-1}, Kawauchi defined the torsion Alexander polynomial of $L$ which we denote it by  ~$\Delta_L^T$. In \cite[Theorems~A,~B]{Kawauchi:1978-1}, he proved that if two links $L_0$ and $L_1$ are concordant, then $\rank_{\Lambda_{\mu}} H_1(X_{L_0};\Lambda_{\mu})=\rank_{\Lambda_{\mu}} H_1(X_{L_1};\Lambda_{\mu})$ and $\Delta^T_{L_0}f_0\overline{f_0}\overset{\cdot}=\Delta_{L_1}^T f_1\overline{f_1}$ for some $f_i(t_1,\ldots,t_\mu)\in \Lambda_{\mu}$, $i=0,1$ with $|f_i(1,\ldots,1)|=1$. 

We extend these theorems of Hillman and Kawauchi in terms of symmetric Whitney tower/grope concordance as follows:
\begin{theoremB}Suppose that two links $L_0$ and $L_1$ are height 3 Whitney tower/grope concordant. Then, $[b_{L_0}]=[b_{L_1}]\in W(\K,\Lambda_{\mu S},-)$.
\end{theoremB}
\begin{theoremC}Suppose that two links $L_0$ and $L_1$ are height 3 Whitney tower/grope concordant. Then,
\begin{enumerate}
\item $\rank_{\Lambda_{\mu}} H_1(X_{L_0};\Lambda_{\mu})=\rank_{\Lambda_{\mu}} H_1(X_{L_1};\Lambda_{\mu})$ and 
\item $\Delta^T_{L_0}f_0\overline{f_0}\overset{\cdot}=\Delta_{L_1}^T f_1\overline{f_1}$ for some $f_i(t_1,\ldots,t_\mu)\in \Lambda_\mu$, $i=0,1$ with $|f_i(1,\ldots,1)|=1$.
\end{enumerate}
\end{theoremC}

As a special case of Theorems~B and C for 2-component links with linking number 1, we have the following special case. This illustrates that the concordance problem between 2-component link with linking number 1 and the Hopf link is similar with the concordance problem between knot and the unknot.
\begin{corollaryD}Suppose that $L$ is a 2-component link with linking number 1 and $H$ is the Hopf link. If $L$ and $H$ are height 3 Whitney tower/grope concordant, then 
\begin{enumerate}
\item $[b_L]=0\in W(\K,\Lambda_2,-)$,
\item $\rank_{\Lambda_2} H_1(X_L;\Lambda_2)=0$, 
\item $\Delta_L^T\overset{\cdot}=f\overline{f}$ 
for some $f(t_1,t_2)\in \Lambda_2$ such that $|f(1,1)|=1$.
\end{enumerate}
\end{corollaryD}

\begin{remark0}Theorems B and C should be compared to the following equivalent statements for knots about abliean invariants. (e.g.\ \cite[Theorem 1.1]{Cochran-Orr-Teichner:2003-1} and \cite[Theorem 5.10]{Cha:2007-1}.)
\begin{enumerate}
\item The knot $K$ bounds a grope of height 2.5 in $D^4$.
\item The 0-surgery manifold of $K$, $M_K$ is 0.5 solvable.
\item The Seifert form of $K$ is metabolic (or $K$ is algebraically slice).
\item The Blanchfield form of $K$ is Witt-trivial.
\end{enumerate}
Therefore, the most natural assumption for Theorems B and C might be the existence of 0.5-solvable cobordism between link exteriors. The proof for the knot case heavily relies on the existence of Seifert surfaces for $K$. For general links, as substitutes for Seifert surfaces, there are \emph{immersed} Cooper surfaces studied in \cite{Cooper:1982-1} (or its generalization in \cite{Cimasoni:2004-1}). However, because of their singularities, the similar approach using Cooper surface seems somewhat difficult.
\end{remark0}

\subsection*{Acknowledgements}
The author would like to express his deep gratitude to his advisor Jae Choon Cha for suggesting the problem and for many valuable conversations about this work. He also would like to thank Stefan Friedl, Jonathan Hillman, Mark Powell for their helpful suggestions. This research was supported by a National Research Foundation of Korea Grant funded by the Korean Government (NRF--2011--0002353).

\section{Casson-Gordon type representations}\label{section:Casson-Gordon-type-representations}
The goal of this section is to prove Theorem \ref{finalproposition} which will give the key dimension estimate in the proof of Theorem A. Lemma \ref{proposition3} and Theorem \ref{finalproposition} are inspired by \cite[Lemma 3.10 and Theorem 3.11]{Cha:2010-01}. In the proof of Lemma \ref{proposition3}, we use the injectivity theorem of Friedl and Powell \cite[Theorem 3.1]{Friedl-Powell:2010-1} stated in Lemma \ref{injectivitytheorem}.  

We recall the notations used in \cite{Friedl-Powell:2010-1} for the convenience of the reader. Let $\varphi\colon G\to A$ be a surjective group homomorphism where $A$ is a finite abelian $p$-group.  Assume that $\varphi\colon G\to A$ factors through a surjective homomorphism $\phi'\colon G\to \H'$ to a torsion free abelian group $\H'$. Let $K=\Ker\varphi$, $\H=\Im(\phi'|_K)$ and $\phi\colon K\to \H$ be the restriction of $\phi'$ to $K$. Note that $\H$ is a torsion free abelian group. 
In short,  we have the following commutative diagram:
\[\begin{diagram}
\node{1}\arrow{e}\node{K}\arrow{s,l}{\phi=\phi'|_{K}}\arrow{e}\node{G}\arrow{e,t}{\varphi}\arrow{s,l}{\phi'}\node{A}\arrow{e}\node{1}\\
\node[2]{\H}\arrow{e}\node{\H'}\arrow{ne}
\end{diagram}
\]
Suppose that $\alpha\colon K\to \GL(d,Q)$ is a $d$-dimensional representation to a field $Q$ of characteristic zero such that $\alpha|_{\Ker\phi}$ factors through a $q$-group for some prime $q$. Let $Q(\H)$ be the quotient field of the group ring $Q[\H]$. Note that $\alpha$ and $\phi$ give the right $\Z K$-module structure on $Q^d\otimes_{Q}Q(\H)=Q(\H)^d$ as follows:
 \begin{equation*}\label{equation:twisted-module}\tag{$\ast$}\begin{aligned} Q^d\mathbin{\mathop{\otimes}_Q}Q(\H)\times \Z K&\to Q^d\mathbin{\mathop{\otimes}_Q}Q(\H)\\
 (v\otimes p,g)&\longmapsto (v\cdot \alpha(g)\otimes p,\phi(g))
 \end{aligned}
 \end{equation*}
 We write $t=|A|$. $\Z G$ is a left (rank $t$ free) $\Z K$-module.  Note that there is a right action of $G$ on $dt$-dimensional $Q$-vector space $Q^d\mathbin{\mathop{\otimes}_{\mathbb{Z}K}}\Z G$. Equivalently, there is an induced representation $\alpha'\colon G\to \GL(dt,Q)$.
 
 As in $\eqref{equation:twisted-module}$, $\alpha'$ and $\phi'$ give the right $\Z G$-module structure on $Q(\H')^{dt}=Q^{dt}\otimes_{Q}Q(\H')$:
\begin{equation*}\label{equation:induced-representation}\tag{$\ast\ast$}
\begin{aligned}
Q^{dt}\mathbin{\mathop{\otimes}_Q}Q(\H')\times \Z G&\to Q^{dt}\mathbin{\mathop{\otimes}_Q}Q(\H')\\
(v\otimes p,g)&\longmapsto (v\cdot \alpha'(g)\otimes p,\phi'(g))
\end{aligned}
\end{equation*}
Regard $\Z/q$ as a $\Z G$-module with the trivial $G$-action. Now, we state \cite[Theorem 3.1]{Friedl-Powell:2010-1}.
\begin{lemma}\cite[Theorem 3.1]{Friedl-Powell:2010-1}\label{injectivitytheorem} Let $f\colon M\to N$ be a morphism of projective left $\Z G$-modules such that 
\[1_{\Z/q}\mathbin{\mathop{\otimes}_{\mathbb{Z}G}} f\colon \Z/q\mathbin{\mathop{\otimes}_{\mathbb{Z}G}}M\to \Z/q\mathbin{\mathop{\otimes}_{\mathbb{Z}G}}N\]
is injective. Then,
\[ 1_{Q(\H')^{dt}}\mathbin{\mathop{\otimes}_{\mathbb{Z}G}}f\colon Q(\H')^{dt}\mathbin{\mathop{\otimes}_{\mathbb{Z}G}}M\to Q(\H')^{dt}\mathbin{\mathop{\otimes}_{\mathbb{Z}G}}N\]
is injective.
\end{lemma}
Using Lemma \ref{injectivitytheorem}, we prove Lemma \ref{proposition3} and Theorem \ref{finalproposition}.   
\begin{lemma}\label{proposition3} Let $f\colon M\to N$ be a morphism of left $\Z G$-modules.
\begin{enumerate}
\item If $N$ is projective, then 
\[ dt\cdot \dim_{\Z/q}\Im (1_{\Z/q}\mathbin{\mathop{\otimes}_{\mathbb{Z}G}}f)\leq \dim_{Q(\H')} \Im (1_{Q(\H')^{dt}}\mathbin{\mathop{\otimes}_{\mathbb{Z}G}}f).\]
\item If, in addition, $M$ is finitely generated and free, then 
\[ dt \cdot \dim_{\Z/q}\Ker (1_{\Z/q}\mathbin{\mathop{\otimes}_{\mathbb{Z}G}}f)\geq \dim_{Q(\H')} \Ker(1_{Q(\H')^{dt}}\mathbin{\mathop{\otimes}_{\mathbb{Z}G}}f).\]
\end{enumerate}
\end{lemma}
\begin{proof}
(1) Let $k=\dim_{\Z/q}\Im(1_{\Z/q}\mathbin{\mathop{\otimes}_{\mathbb{Z}G}} f)$. ($k$ may be any cardinal number.) Note that $\Z/q\mathbin{\mathop{\otimes}_{\mathbb{Z}G}}-$ induces two surjections $(\Z G)^k\to \Z/q^k$ and $\Im f\to \Im (1_{\Z/q}\mathbin{\mathop{\otimes}_{\mathbb{Z}G}} f)$. Since $(\Z G)^k$ is free, the following diagram commutes:
\[\begin{diagram}
\node{(\Z G)^k}\arrow{e,t,..}{\exists~i}\arrow{s}\node{\Im f}\arrow{s}\\
\node{\Z/q^k}\arrow{e,t}{\cong}\node{\Im (1_{\Z/q}\mathbin{\mathop{\otimes}_{\mathbb{Z}G}} f)}
\end{diagram}
\]
Recall that $N$ is a projective $\Z G$-module. Obviously, $(\Z G)^k$ is a projective $\Z G$-module. Hence, we can apply Lemma \ref{injectivitytheorem} to $i\colon (\Z G)^k\to \Im f\subset N$ and obtain the following injection:
\[1_{Q(\H')^{dt}}\mathbin{\mathop{\otimes}_{\mathbb{Z}G}} i\colon Q(\H')^{dtk}=Q(\H')^{dt}\mathbin{\mathop{\otimes}_{\mathbb{Z}G}} (\Z G)^k\hookrightarrow Q(\H')^{dt}\mathbin{\mathop{\otimes}_{\mathbb{Z}G}} N\]
Since $\Im i\subset \Im f$, 
\[Q(\H')^{dtk}\cong \Im(1_{Q(\H')^{dt}}\mathbin{\mathop{\otimes}_{\mathbb{Z}G}} i)\subset \Im (1_{Q(\H')^{dt}}\mathbin{\mathop{\otimes}_{\mathbb{Z}G}} f).\] This implies 
\[ dt\cdot \dim_{\Z/q}\Im(1_{\Z/q}\mathbin{\mathop{\otimes}_{\mathbb{Z}G}} f)= dtk\leq  \dim_{Q(\H')}(1_{Q(\H')^{dt}}\mathbin{\mathop{\otimes}_{\mathbb{Z}G}} f).\]
(2) Let $M=(\Z G)^n$. (1) and the following elementary observation completes the proof.
\[ \dim_{\Z/q}\Ker (1_{\Z/q}\mathbin{\mathop{\otimes}_{\mathbb{Z}G}} f)+\dim_{\Z/q}\Im(1_{\Z/q}\mathbin{\mathop{\otimes}_{\mathbb{Z}G}} f)=\dim_{\Z/q}\Z/q\mathbin{\mathop{\otimes}_{\mathbb{Z}G}} M=n\]
and similarly,
\[ ndt=\dim_{Q(\H')}\Ker(1_{Q(\H')^{dt}}\mathbin{\mathop{\otimes}_{\mathbb{Z}G}} f) +\dim_{Q(\H')}\Im(1_{Q(\H')^{dt}}\mathbin{\mathop{\otimes}_{\mathbb{Z}G}} f).\]
\end{proof}
\begin{theorem}\label{finalproposition} Suppose $C_*$ is a chain complex of projective left $\Z G$-modules with  $C_n$ finitely generated.  If $\{x_i\}_{i\in I}$ is a collection of $n$-cycles in $C_n$, then for the $(Q(\H)^d\otimes_{\Z K}\Z G)$-span of $\{[1_{Q(\H)^{d}}\mathbin{\mathop{\otimes}_{\mathbb{Z}K}} x_i]\}_{i\in I}$, $M\subset H_n(Q(\H)^{d}\mathbin{\mathop{\otimes}_{\mathbb{Z}K}} C_*)$ and the $\Z/q$-span of $\{[1_{\Z/q}\mathbin{\mathop{\otimes}_{\mathbb{Z}G}}x_i ]\}_{i\in I}$, $\overline{M}\subset H_n(\Z/q\mathbin{\mathop{\otimes}_{\mathbb{Z}G}}C_*)$, we have
\[ \dim_{Q(\H)}H_n(Q(\H)^{d}\mathbin{\mathop{\otimes}_{\mathbb{Z}K}}C_*)/M\leq dt\cdot \dim_{\Z/q}H_n(\Z/q\mathbin{\mathop{\otimes}_{\mathbb{Z}G}}C_*)/\overline{M}.
\]
\end{theorem}
\begin{proof}Let $\partial_n\colon C_n\to C_{n-1}$ be the boundary map of $C_*$ and define $f\colon (\Z G)^I\oplus C_{n+1}\to C_n$ by $(e_i,v)\mapsto x_i+ \partial_{n+1}(v)$, where $\{e_i\}_{i\in I}$ is the standard basis of $(\Z G)^I$. Then, 
\[ H_n(Q(\H)^{d}\mathbin{\mathop{\otimes}_{\mathbb{Z}K}} C_*)/M=\Ker(1_{Q(\H)^{d}} \mathbin{\mathop{\otimes}_{\mathbb{Z} K}}\partial_n )/\Im(1_{Q(\H)^{d}}\mathbin{\mathop{\otimes}_{\mathbb{Z} K}} f )\textrm{~and~}
\]
\[H_n(\Z/q\mathbin{\mathop{\otimes}_{\mathbb{Z}G}}C_*)/\overline{M}=\Ker(1_{\Z/q}\mathbin{\mathop{\otimes}_{\mathbb{Z}G}} \partial_n)/\Im(1_{\Z/q}\mathbin{\mathop{\otimes}_{\mathbb{Z}G}}f).
\]
From the $\Z G$-module structure on $Q(\H')^{dt}$ in \eqref{equation:induced-representation},
 \[ Q(\H')\mathbin{\mathop{\otimes}_{Q(\H)}}  Q(\H)^d\mathbin{\mathop{\otimes}_{\Z K}}\Z G = (Q(\H')\mathbin{\mathop{\otimes}_{Q(\H)}}  Q(\H)^d\mathbin{\mathop{\otimes}_{\Z K}}\Z G) \mathbin{\mathop{\otimes}_{\Z G}}\Z G=Q(\H')^{dt}\mathbin{\mathop{\otimes}_{\Z G}}\Z G.\]
 Since $C_*$ is a chain complex of left $\Z G$-modules,  \[Q(\H')\mathbin{\mathop{\otimes}_{Q(\H)}} Q(\H)^d\mathbin{\mathop{\otimes}_{\Z K}} C_*=Q(\H')^{dt}\mathbin{\mathop{\otimes}_{\Z G}}C_*.\]
Since $\H\hookrightarrow \H'$, $Q(\H')$ is a flat right $Q (\H)$-module. Therefore, we have
 \[H_*(Q(\H')^{dt}\mathbin{\mathop{\otimes}_{\Z G}}C_*)=Q(\H')\mathbin{\mathop{\otimes}_{Q(\H)}} H_*(Q(\H)^d\mathbin{\mathop{\otimes}_{\Z K}} C_*).\]
Combining these, we obtain 
\[\dim_{Q(\H)}H_n(Q(\H)^{d}\mathbin{\mathop{\otimes}_{\mathbb{Z}K}}C_*)/M=\dim_{Q(\H')}H_n(Q(\H')^{dt}\mathbin{\mathop{\otimes}_{\mathbb{Z}G}} C_*)/(Q(\H')\mathbin{\mathop{\otimes}_{Q(\H)}}M)\]
and 
\[H_n(Q(\H')^{dt}\mathbin{\mathop{\otimes}_{\mathbb{Z}G}} C_*)/(Q(\H')\mathbin{\mathop{\otimes}_{Q(\H)}}M)=\Ker(1_{Q(\H')^{dt}} \mathbin{\mathop{\otimes}_{\mathbb{Z} G}}\partial_n )/\Im( 1_{Q(\H')^{dt}}\mathbin{\mathop{\otimes}_{\mathbb{Z} G}} f ).\]
From the above observations and the inequality from Lemma \ref{proposition3},
\begin{eqnarray*}
&&\dim_{Q(\H)}H_n(Q(\H)^d\mathbin{\mathop{\otimes}_{\mathbb{Z}K}}C_*)/M\\
&=&\dim_{Q(\H')}H_n(Q(\H')^{dt}\mathbin{\mathop{\otimes}_{\mathbb{Z}G}} C_*)/(Q(\H')\mathbin{\mathop{\otimes}_{Q(\H)}}M)\\
&=&\dim_{Q(\H')}\Ker(1_{Q(\H')^{dt}}\mathbin{\mathop{\otimes}_{\mathbb{Z}G}}\partial_n)-\dim_{Q(\H')}\Im(1_{Q(\H')^{dt}}\mathbin{\mathop{\otimes}_{\mathbb{Z}G}} f)\\
&\leq& dt(\dim_{\Z/q}\Ker (1_{\Z/q} \mathbin{\mathop{\otimes}_{\mathbb{Z}G}}\partial_n)-\dim_{\Z/q}\Im (1_{\Z/q}\mathbin{\mathop{\otimes}_{\mathbb{Z}G}} f))\\
&=&dt\cdot \dim_{\Z/q} H_n(\Z/q\mathbin{\mathop{\otimes}_{\mathbb{Z}G}}C_*)/\overline{M}.\end{eqnarray*}
This completes the proof.
\end{proof}

\section{$h$-solvable cobordism}\label{solvablecobordism}
In this section, we give the definition of an $h$-solvable cobordism following \cite{Cha:2012-1}. Also, we prove Proposition \ref{proposition1} about prime power covering of 1-solvable cobordism.  
\subsection{Definition of $h$-solvable cobordism}\label{subsection:definition-solvable-cobordism}
For oriented compact bordered 3-manifolds $M$ and $M'$ with a chosen homeomorphism $\partial M\cong \partial M'$, a \emph{cobordism $W$ between $M$ and $M'$} is a $4$-dimensional manifold with boundary $\partial W=M\cup_\partial -M'$, where $-M'$ denotes $M'$ with reversed orientation. We often denote a cobordism by $(W;M,M')$. A cobordism $(W;M,M')$ is an \emph{$H_1$-cobordism} (resp.\ a \emph{homology cobordism}) if $H_i(M;\Z)\cong H_i(W;\Z)\cong H_i(M';\Z)$ under the inclusion map for $i\leq 1$ (resp.\ for all $i$). Note  that $H_2(W,M)$ is a free abelian group if $(W;M,M')$ is an $H_1$-cobordism (for example, see \cite[Lemma 3.7]{Cha:2010-01}).
\begin{example0} If $L$ is a link in $S^3$, then the link exterior $X_L$ is a bordered 3-manifold with a canonical homeomorphism between disjoint union of tori and $\partial X_L$ sending standard basis to the meridians and $0$-framed longitudes of $L$. If two links $L$ and $L'$ are concordant, then $X_L$ and $X_{L'}$ are homology cobordant bordered 3-manifolds via a concordance exterior and $h$-solvable cobordant for all $h\in \frac{1}{2}\Z_{\geq 0}$ (see the definition of solvable cobordism given in~below).
\end{example0}
We use the following notation for covering maps associated to the derived series.
\\
\begin{convention}~
\begin{enumerate}
\item For a space $X$, there is a sequence of regular covers over $X$ 
\[ X^{(n+1)}\to X^{(n)}\to \cdots \to X^{(1)}\to X^{(0)}=X\]
which corresponds to the derived series 
\[ \pi^{(n+1)}\subset \pi^{(n)}\subset\cdots\subset \pi^{(0)}=\pi, \textrm{~where~} \pi=\pi_1(X) \textrm{~and~} \pi^{(n+1)}=[\pi^{(n)},\pi^{(n)}].\] 
With this, we can always identify $H_*(X;\Z[\pi/\pi^{(n)}])=H_*(X^{(n)};\Z)$ as usual.
\item For a 4-manifold $W$ with $\pi=\pi_1(W)$, let
\[ \lambda_n\colon H_2(W;\Z[\pi/\pi^{(n)}])\times H_2(W;\Z[\pi/\pi^{(n)}])\to \Z[\pi/\pi^{(n)}]\] be the inetersection form.
\item For a covering map $Y\to X$, $\Cov(Y|X)$ denotes its deck transformation group. Assume that the action of $\Cov(Y|X)$ on $H_*(Y;\Z)$ is a right action.
\end{enumerate}
\end{convention}
\begin{definition}Suppose $(W;M,M')$ is an $H_1$-cobordism between bordered 3-manifolds $M$ and $M'$ with $\pi=\pi_1(W)$. Let $r=\frac{1}{2}\rank H_2(W,M;\Z)$. 
\begin{enumerate}
\item A submodule $L\subset  H_2(W;\Z[\pi/\pi^{(n)}])$ is an \emph{$n$-lagrangian} if $L$ projects to a half-rank summand of $H_2(W,M;\Z)$ and $\lambda_n$ vanishes on $L$.
\item For an $n$-lagrangian $L$ $(k\leq n$), homology classes $d_1,\ldots, d_r\in H_2(W;\Z[\pi/\pi^{(k)}])$ are \emph{$k$-duals} if $L$ is generated by $l_1,\ldots,l_r\in L$ whose projections $l_1',\ldots,l_r'\in H_2(W;\Z[\pi/\pi^{(k)}])$ satisfy $\lambda_k(l_i',d_j)=\delta_{ij}$. 
\item An $H_1$-cobordism $(W;M,M')$ is called an \emph{$n.5$-solvable cobordism} (resp.\ \emph{$n$-solvable cobordism}) if it has an $(n+1)$-lagrangian (resp.\ $n$-lagrangian) admitting $n$-duals. If there exists an $h$-solvable cobordism from $M$ to $M'$, we say that $M$ is $h$-solvable cobordant to $M'$ for $h\in\frac{1}{2}\Z_{\geq 0}$.
\end{enumerate}
\end{definition}
\subsection{Prime power cover of 1-solvable cobordism}\label{technicallemmas}
In this subsection, we prove Proposition \ref{proposition1} about (abelian) prime power cover of 1-solvable cobordism for later purpose.
\begin{proposition}\label{proposition1} Suppose that $(W;M,M')$ is a 1-solvable cobordism with $\varphi\colon \pi_1W\to A$ be a surjective group homomorphism to an abelian $p$-group $A$ and $p$ is prime. We denote the cobordism of the induced coverings by $(W^\varphi;M^\varphi,M'^\varphi)$. Then, 
\begin{enumerate}
\item $\beta_2(W^\varphi,M^\varphi)=|A|\beta_2(W,M)$ where $\beta_2$ is the second Betti number.
\item The inclusion induced map $FH_2(W^\varphi;\Z)\to FH_2(W^\varphi,M^\varphi;\Z)$ is surjective.  
\item  $(W^\varphi;M^\varphi,M'^\varphi)$ is an $H_1$-cobordism with $\Q$-coefficients.
\end{enumerate}
Here, for a finitely generate abelian group $G$, $FG$ denotes the free part of $G$.
\end{proposition}
\begin{proof} (1) Fix a (relative) CW-complex structure of $(W,M)$. This induces a (relative) CW-complex structure of $(W^\varphi,M^\varphi)$. Let $C_*=C_*(W^\varphi,M^\varphi;\Z)$. Then, $C_*$ is a chain complex of free $\Z A$-modules and $C_*(W,M;\Z)=C_*\otimes_{\Z A} \Z$. Since $(W;M,M')$ is an $H_1$-cobordism, $H_i(C_*\otimes_{\Z A}\Z/p)=0\textrm{~for~}i=0,1$ by universal coefficient theorem. Since $p$ is prime, the well-known Levine's chain homotopy lifting argument in \cite{Levine:1994-1} shows that $H_i(C_*\otimes_\Z \Z/p)=0$ for $i=0,1$.  In particular, by universal coefficient theorem, $H_i(C_*)$ is a torsion abelian group for $i=0,1$. By universal coefficient theorem, $H_i(W^\varphi,M^\varphi;\Q)=H_i(C_*)\otimes_{\Z}\Q=0$ for $i=0,1$. 

By taking $C_*=C_*(W^\varphi,M'^\varphi;\Z)$, the same argument shows $H_i(W^\varphi,M'^\varphi;\Q)=0$ for $i=0,1$. By Poincar\'{e} duality and universal coefficient theorem, 
\[H_i(W^\varphi,M^\varphi;\Q)\cong\Hom_\Q(H_{4-i}(W^\varphi,M'^\varphi;\Q),\Q)=0\textrm{~for~}i=3,4.\] So, $\beta_2(W^\varphi,M^\varphi)=\chi(W^\varphi,M^\varphi)$ where $\chi$ is the Euler characteristic. Similarly,  $\chi(W,M)=\beta_2(W,M)$ because $H_i(W,M;\Z)=H_i(W,M';\Z)=0$ for $i=0,1$. By definition, $(W^\varphi,M^\varphi)$ is an $A$-cover of $(W,M)$ and $\chi(W^\varphi,M^\varphi)=|A|\chi (W,M)$. This completes the proof of (1).

(2) Since $W^\varphi\to W$ is an abelian covering with $\Cov(W^\varphi|W)=A$,  $\pi_1(W)^{(1)}\subset \pi_1(W^\varphi)$. The covering map $W^{(1)}\to W^\varphi$ induces $H_2(W^{(1)};\Z)\to H_2(W^\varphi;\Z)$. Let  $l_1,\ldots,l_r,d_1,\ldots,d_r$ be the images of the (generators of) $1$-lagrangian and $1$-duals in $H_2(W^\varphi;\Z)$. By the definition of 1-solvable cobordism, $\beta_2(W,M)=2r$. Let $A=\{g_1,\ldots,g_t\}$.  From (1), $\beta_2(W^\varphi,M^\varphi)=\beta_2(W,M)|A|=2rt$.

From the (right) group action of $A$ on $H_2(W^\varphi;\Z)$, we can define
\[ l_{ij}=l_i\cdot g_j\textrm{~and~} d_{kl}=d_k\cdot g_l \textrm{~for~} 1\leq i,k\leq r\textrm{~and~}1\leq j,l\leq t.\] By the definition of 1-Lagrangian and 1-duals, the intersection pairing $\lambda\colon FH_2(W^\varphi,M^\varphi;\Z)\times FH_2(W^\varphi,M'^\varphi;\Z)\to \Z$ restricted to the $\Z$-span of (the image of) $\{l_{ij},d_{kl}\}$ is \[\begin{pmatrix}
        0&I_{rt\times rt}\\
        I_{rt\times rt}&X
     \end{pmatrix}\]
By rank counting, the image of $\{l_{ij},d_{kl}\}$ form a $\Z$-basis of $FH_2(W^\varphi,M^\varphi;\Z)$. This proves that $\operatorname{inc}_*\colon FH_2(W^\varphi;\Z)\xrightarrow{} FH_2(W^\varphi,M^\varphi;\Z)$ is surjective because $\{l_{ij},d_{kl}\}\subset FH_2(W^\varphi;\Z)$.

(3) From (2), $\operatorname{inc}_*\colon H_2(W^\varphi;\Q)\to H_2(W^\varphi,M^\varphi;\Q)$ is surjective.  From (1) and the homology long exact sequence of a pair $(W^\varphi,M^\varphi)$, $\operatorname{inc}_*\colon H_i(M^\varphi;\Q)\to H_i(W^\varphi;\Q)$ is an isomorphism for $i=0,1$. Same argument works for $(W,M')$. This completes the proof.

\end{proof}

\section{Solvable cobordism and Friedl-Powell invariant}\label{existenceofmetabolizer} 
Throughout this section, for any finitely generated abelian group $G$, $tG$ and $FG$ denote the torsion part of $G$ and free part of $G$, respectively.  $G^\wedge$ denotes $\Hom_\Z(G,\Q/\Z)$. For a finite abelian group $G$, $G^\wedge=\Ext_\Z(G,\Z)$ since $\Hom_\Z(G,\Q)=\Ext_\Z(G,\Q)=0$. $H_*(-)$ denotes homology with integral coefficients.

\subsection{Definition of the Friedl-Powell invariant $\tau(L,\chi)$}
To define the Friedl-Powell invariant $\tau(L,\chi)$, we set up notations and conventions used in \cite{Friedl-Powell:2011-1}. Here, $L$ is a 2-component link with linking number 1 and $H$ is the Hopf link. We denote the exterior of $J$ by $X_J=S^3-\nu(J)$  for $J=H,L$.  We can decompose $\partial X_L$ into $Y_a\cup Y_b$ with $Y_a\cong Y_b=S^1\times D^1\,\sqcup\,S^1\times D^1$ where both $Y_a$ and $Y_b$ are annuli neighborhood of (parallels of) meridians of  $L$. Define $M_J=X_J\,\cup_{\partial X_H\times I}\, -X_H$ for $J=H,L$ where the gluing map respects the ordering of the link components and identifies each of the subsets $Y_a, Y_b\subset \partial X_J$ for $J=L,H$. 

For a prime $p$, we say a group homomorphism $\varphi\colon H_1(M_L)\to \Z/p^i\oplus \Z/p^j$ is \emph{admissible} if $\varphi$ sends two meridians of $L$ to the standard basis $(1,0), (0,1)$. (From the Mayer-Vietoris sequence, $H_1(M_L)\cong H_1(X_L)\+ \Z\cong \Z^3$.) Let $M_L^\varphi\to M_L^{\vphantom{\varphi}}$ be the $p^{i+j}$-fold covering space from $\varphi$. We denote the Hurewicz map by $\phi'\colon \pi_1(M_L)\to H_1(M_L)$. Define $\phi\colon \pi_1(M_L^\varphi)\to H_1(M_L^{\vphantom{\varphi}})$ be the restriction $\phi'|_{\pi_1(M_L^\varphi)}$ and $\H=\Im{\phi}$. Choose an isomorphism $\psi\colon \pi_1(T^3)\cong \H$. (Note that $\H$ is isomorphic to $\Z^3$ as a finite index subgroup of $H_1(M_L)\cong \Z^3$.)

For a prime power character $\chi\colon \pi_1(M_L^\varphi)\to \Z/q^k$, we have the bordism class 
\[[(M_L^\varphi,\chi\times \phi)\sqcup -(T^3,\operatorname{tr}\times \psi)]\in \Omega_3(\Z/q^k\times \H)\] where $\operatorname{tr}\colon \pi_1(T^3)\to \Z/q^k$ is the trivial group homomorphism. From the Atiyah-Hirzebruch spectral sequence calculation given in \cite[Section 3.2]{Friedl-Powell:2011-1}$,[(M_L^\varphi,\chi\times \phi)\,\sqcup\, -(T^3,\operatorname{tr}\times \psi)]$ is $q$-primary torsion in $\Omega_3(\Z/q^k\times \H)$. In other words, there exist a non-negative integer $s$, a cobordism $W$ between $q^sM_L^\varphi$ and $q^{s}T^3$, and $\Phi\colon \pi_1(W)\to \Z/q^k\times\H$ such that the following diagram commutes:
\[
\begin{diagram}\dgARROWLENGTH=2.5em
      \node{\bigsqcup^{q^{s}}\,(M_L^\varphi\,\sqcup\,- T^3)}\arrow{s,l,J}{\partial}\arrow[3]{e,t}{\chi\times \phi\,\sqcup\, \operatorname{tr}\times \psi}\node[3]{K(\Z/q^k\times\H,1)} 
        \\
      \node{W} \arrow{neee,b}{\Phi} 
    \end{diagram}
\]
From the following sequence of ring homomorphisms,
\[\Z[\pi_1(W)]\xrightarrow{\Phi} \Z[\Z/q^k\times \H]=\Z[\Z/q^k][\H]\to \Q(\xi_{q^k})(\H)\to \mathbb{C}(\H)=\mathcal{K}\]
we can define the  twisted intersection form $H_2(W;\K)\times H_2(W;\K)\to \K$. We denote the non-singular part of the intersection form on $H_2(W;\K)$ (resp.\ on $H_2(W;\Q)$) by $\lambda_\K(W)$ (resp.\ $\lambda_\Q(W)$). 
\begin{definition}[Friedl-Powell invariant]\label{definition:Friedl-Powell-invariant}\[\tau(L,\chi)=(\lambda_\K(W)-\K\otimes \lambda_\Q(W))\otimes\frac{1}{q^s}\in L^0(\K)\mathbin{\mathop{\otimes}_{\Z}}\Z[1/q]\]
where $L^0(\K)$ is the Witt group of finite dimensional non-singular sesquilinear forms over $\K$.
\end{definition}
In \cite[Section 3.2]{Friedl-Powell:2011-1}, it is shown that $\tau(L,\chi)$ is well-defined. That is,  $\tau(L,\chi)$ depend neither on the choice of $W$ nor on the choice of isomorphism $\psi\colon \pi_1(T^3)\to \H$.

In Section \ref{subsection:metabolizer}, we will describe the linking form 
\[\lambda_L\colon tH_1(X_L^\varphi,Y_a^\varphi)\times tH_1(X_L^\varphi,Y_a^\varphi)\to \Q/\Z\]
Now, we give the precise statement that \emph{the Friedl-Powell invariant vanishes}.
\begin{definition}\label{definition:vanishing-Friedl-Powell-invariant}For 2-component link $L$ with linking number 1, we say \emph{the Friedl-Powell invariant vanishes for $L$} if for any admissible homomorphism $\varphi\colon H_1(M_L)\to \Z/p^i\otimes \Z/p^j$ and for a prime $p$, there exists a metabolizer $P=P^\perp$ of the linking form 
\[\lambda_L\colon tH_1(X_L^\varphi,Y_a^\varphi)\times tH_1(X_L^\varphi,Y_a^\varphi)\to \Q/\Z\] with the following property:  
for any character of prime power order $\chi\colon H_1(M_L^\varphi)\to \Z/q^k$ which satisfies that $\chi|_{H_1(X_L^\varphi)}$ factors through
\[
\begin{diagram}\dgARROWLENGTH=2.5em
      \node{H_1(X_L^\varphi)}\arrow{s}\arrow[2]{e,t}{ \chi|_{H_1(X_L^\varphi)}}\node[2]{\Z/q^k} 
        \\
      \node{H_1(X_L^\varphi,Y_a^\varphi)} \arrow{nee,b}{\delta} 
    \end{diagram}
\]
and that $\delta$ vanishes on $P$,  $\tau(L;\chi)=0\in L^0(\K)\otimes_\Z\Z[1/q]$.
\end{definition}
The following main theorem will be proved in Section \ref{subsection:proof-of-main-theorem}:
\begin{theoremA}\label{maintheorem} Suppose that $L$ is a 2-component link  with linking number 1 and $H$ is the Hopf link. If $X_L$ and $X_H$ are 1.5-solvable cobordant, then the Friedl-Powell invariant $\tau(L,\chi)$ vanishes for $L$ in the sense of Definition \ref{definition:vanishing-Friedl-Powell-invariant}. In particular, the conclusion holds if $L$ and $H$ are height 3.5 Whitney tower/grope concordant. 
\end{theoremA}

\subsection{1-solvable cobordism and a metabolizer of the linking form}\label{subsection:metabolizer}
In this subsection, we recall the definition of the linking form 
\[\lambda_L\colon tH_1(X_L^\varphi,Y_a^\varphi)\times tH_1(X_L^\varphi,Y_a^\varphi)\to \Q/\Z\] defined in \cite{Friedl-Powell:2011-1} and prove Proposition \ref{proposition:metabolizer}. The adjoint of $\lambda_L$, $\operatorname{Ad}(\lambda_L)\colon tH_1(X_L^\varphi,Y_a^\varphi)\to tH_1(X_L^\varphi,Y_a^\varphi)^\wedge$, can be obtained by composing the following isomorphisms:
\begin{equation}\label{definitionoflinkingform}\tag{a}~tH_1(X_L^\varphi,Y_a^\varphi)\to tH^2(X_L^\varphi,Y_b^\varphi)\to\Ext_\Z (tH_1(X_L^\varphi,Y_b^\varphi),\Z)=tH_1(X_L^\varphi,Y_a^\varphi)^\wedge.\end{equation}
We used Poincar\'{e} duality, universal coefficient theorem, $H_1(X_L^\varphi,Y_b^\varphi)\cong H_1(X_L^\varphi,Y_a^\varphi)$, and $tH_1(X_L^\varphi,Y_a^\varphi)$ is a finite abelian group. 

Let $(W_0;X_L,X_H)$ be a 1-solvable cobordism. Recall $\varphi\colon H_1(M_L)\to \Z/p^i\oplus \Z/p^j$ is an admissible homomorphism and $H_1(M_L)\cong H_1(X_L)\,\oplus\, \Z$. Then, $\varphi|_{H_1(X_L)}$ extends to $H_1(W_0)$ canonically because $H_1(X_L)\cong H_1(W_0)$. In this sense, an admissible homomorphism $\varphi$ always induces a covering $(W_0^\varphi;X_L^\varphi,X_H^\varphi)\to (W;X_L,X_H)$.
\begin{proposition}\label{proposition:metabolizer} Suppose $(W_0;X_L,X_H)$ is a 1-solvable cobordism. Let $(W_0^\varphi;X_L^\varphi,X_H^\varphi)$ be a covering induced from an admissible homomorphism $\varphi\colon H_1(M_L)\to \Z/p^i\oplus \Z/p^j$, then 
\[ P=\Ker(tH_1(X_L^\varphi,Y^\varphi_a)\to tH_1(W_0^\varphi,Y^\varphi_a))\] is a metabolizer of the linking form $\lambda_L$. 
\end{proposition}
\begin{proof}Suppose that we have the diagram (\ref{maindiagram}) with two exact rows. Then, \[P^\perp = (\operatorname{Ad}(\lambda_L))^{-1}(\Ker \d^\wedge)=\Ker (\operatorname{inc}_*)=P.\] Hence, it suffices to prove the existence of diagram (\ref{maindiagram}) with two exact rows.
\begin{equation}\label{maindiagram}\tag{b}
\begin{diagram}
\node{tH_2(W_0^\varphi,X_L^\varphi)}\arrow{e,t}{\partial}\arrow{s,l}{\theta_1}\arrow{s,r}{\cong}\node{tH_1(X_L^\varphi,Y^\varphi_a)}\arrow{e,t}{\operatorname{inc}_*}\arrow{s,l}{\operatorname{Ad}(\lambda_L)}\arrow{s,r}{\cong}\node{tH_1(W_0^\varphi,Y^\varphi_a)}\arrow{s,l}{\theta_2}\arrow{s,r}{\cong}
\\\node{tH_1(W_0^\varphi,Y_a^\varphi)^\wedge}\arrow{e,t}{\operatorname{inc}_*^\wedge}\node{tH_1(X_L^\varphi,Y_a^\varphi)^\wedge}\arrow{e,t}{\partial^\wedge}\node{tH_2(W_0^\varphi,X_L^\varphi)^\wedge}
\end{diagram}
\end{equation}
As in (\ref{definitionoflinkingform}), $tH_2(W_0^\varphi,X_L^\varphi)\cong tH_1(W_0^\varphi,X_H^\varphi)^\wedge$ and $tH_1(W_0^\varphi,Y_a^\varphi)\cong tH_2(W_0^\varphi,\partial W_0^\varphi-Y_a^\varphi)^\wedge$ by Ponicar\'{e} duality and universal coefficient theorem. (Note that $\d W_0^\varphi=X_L^\varphi\,\cup\,X_H^\varphi$.) So, from the following claim, we can define isomorphisms $\theta_1$ and $\theta_2$.
\begin{claim}The inclusion maps induce two isomorphisms 
\begin{enumerate}
\item $tH_1(W_0^\varphi,Y^\varphi_a)\cong tH_1(W_0^\varphi,X^\varphi_H)$ and
\item $tH_2(W_0^\varphi,X_L^\varphi)\cong tH_2(W_0^\varphi,\partial W_0^\varphi-Y_a^\varphi)$.
\end{enumerate}
\end{claim}
\begin{proof}[Proof of Claim]
By Proposition \ref{proposition1} (3), $(W_0^\varphi;X_L^\varphi,X_H^\varphi)$ is an $H_1$-cobordism with $\Q$-coefficients. From this and the proofs of \cite[Lemmas 2.6, 2.7 and 2.9]{Friedl-Powell:2011-1} ($W_0$ plays the role of $E_C$),  
\[\Coker(\operatorname{inc}_*\colon H_1(X_H^\varphi,Y_a^\varphi)\to H_1(W_0^\varphi,Y_a^\varphi))\cong tH_1(W_0^\varphi,Y_a^\varphi).\]
From the homology long exact sequence of a triple $(W_0^\varphi,X_H^\varphi,Y_a^\varphi)$, we have an exact sequence
\[ 0\to tH_1(W_0^\varphi,Y_a^\varphi)\to H_1(W_0^\varphi,X_H^\varphi)\to H_0(X_H^\varphi,Y_a^\varphi)=0\]
which proves (1).

From the proof of \cite[Lemma 2.5]{Friedl-Powell:2011-1}, $tH_1(\d W_0^\varphi-Y_a^\varphi,X_L^\varphi)=0$ and the inclusion map $(\d W_0^\varphi-Y_a^\varphi,X_L^\varphi)\to (\d W_0^\varphi, X_L^\varphi)$ induces the zero map on $H_2$.  In particular,  
\[\operatorname{inc}_*\colon H_2(\d W_0^\varphi-Y_a^\varphi,X_L^\varphi)\to  H_2( W_0^\varphi, X_L^\varphi)\]
is also the zero map.  From the homology long exact sequence of a  triple $(W_0^\varphi,\d W_0^\varphi-Y_a^\varphi,X_L^\varphi)$,
\[H_2(W_0^\varphi,X_L^\varphi)\cong\Ker (\d \colon H_2(W_0^\varphi,\d W_0^\varphi-Y_a^\varphi)\to H_1(\d W_0^\varphi-Y_a^\varphi,X_L^\varphi)).\]
By taking torsion subgroups, we obtain (2) via
\begin{eqnarray*} tH_2(W_0^\varphi,X_L^\varphi)&\cong&\Ker (tH_2(W_0^\varphi,\d W_0^\varphi-Y_a^\varphi)\to tH_1(\d W_0^\varphi-Y_a^\varphi,X_L^\varphi)=0)\\
&=&tH_2(W_0^\varphi,\d W_0^\varphi-Y_a^\varphi).
\end{eqnarray*}
\end{proof}
Commutativity of the diagram (\ref{maindiagram}) also easily follows. For exactness of the first row of (\ref{maindiagram}), we prove the following Lemma.
\begin{lemma}\label{exactlemma} Suppose $(W_0;X_L,X_H)$ is a 1-solvable cobordism. We have the following exact sequence 
\[
\begin{diagram}
\node{tH_2(W_0^\varphi,X_L^\varphi)}\arrow{e,t}{\partial}\node{tH_1(X_L^\varphi,Y_a^\varphi)}\arrow{e,t}{\operatorname{inc}_*}\node{tH_1(W_0^\varphi,Y_a^\varphi)}
\end{diagram}
\]
which is the restriction of a long exact sequence of triple $(W_0^\varphi,X_L^\varphi,Y_a^\varphi)$ to their torsion subgroups.
\end{lemma}
\begin{proof}[Proof of the Lemma \ref{exactlemma}] Since $\operatorname{inc}_*\circ\partial =0$, we prove that $\Ker(\operatorname{inc}_*)\subset \Im \partial$.    
  Let $x\in \Ker(\operatorname{inc}_*)$. By the homology long exact sequence of triple $(W_0^\varphi,X_L^\varphi,Y_a^\varphi)$, there exists $y\in H_2(W_0^\varphi,X_L^\varphi)$ such that $\partial y=x$. By Proposition \ref{proposition1} (2), $FH_2(W_0^\varphi)\to FH_2(W_0^\varphi,X_L^\varphi)$ is surjective. So, $j\colon FH_2(W_0^\varphi,Y_a^\varphi)\to FH_2(W_0^\varphi,X_L^\varphi)$ is also surjective. We can choose $z\in FH_2(W_0^\varphi,Y_a^\varphi)$ such that $y-j(z)\in tH_2(W_0^\varphi,X_L^\varphi)$. Then, $\partial (y-j(z))=\partial y=x$ and this shows that $\Ker(\operatorname{inc}_*)\subset \Im \partial$.
\end{proof}
Note that if $A\xrightarrow{f} B\xrightarrow{g} C$ is an exact sequence of abelian groups. Since $\Q/\Z$ is a divisible group, $\Q/\Z$ is an injective $\Z$-module. For any abelian group $G$, $\Ext_\Z(G,\Q/\Z)=0$. Hence, $\Hom_\Z(-,\Q/\Z)$ is an exact functor and we obtain $C^\wedge\to B^\wedge\to A^\wedge$ is exact. This proves that the second row of the diagram (\ref{maindiagram}) is also exact and completes the proof of Proposition \ref{proposition:metabolizer}.\end{proof}
\subsection{Proof of Theorem A}\label{subsection:proof-of-main-theorem}
In this subsection, we prove Theorem A.  Let $(W_0;X_L,X_H)$ be a 1.5-solvable cobordism with $\beta_2(W_0,X_L)=2r$. Note that $\partial W_0=X_L\cup\partial X_H\times I \cup -X_H=M_L$. Attach $X_H\times I$ to $W_0$ along $\partial X_H\times I$ to get 
\[W=W_0\,\cup_{\partial X_H\times I}\,X_H\times I\] with $\partial W=M_L\,\sqcup\, -M_H$. Recall $\varphi\colon H_1(M_L)\to \Z/p^i\oplus \Z/p^j$. Applying Mayer-Vietoris argument to $W=W_0\,\cup\, X_H\times I$, the inclusion induces $H_1(M_L)\cong H_1(W)$. So, $\varphi$ extends to $H_1(W)$ naturally and denote the induced cobordism of coverings by $(W_{\vphantom{L}}^\varphi;M_L^\varphi,M_H^\varphi)$.

From Proposition~\ref{proposition:metabolizer}, we can take a metabolizer 
\[P:=\Ker(tH_1(X_L^\varphi,Y^\varphi_a)\to tH_1(W_0^\varphi,Y^\varphi_a))\] of the linking form $\lambda_L$.  We fix a character $\chi \colon H_1(M_L^\varphi)\to \Z/q^k$ satisfies that $\chi|_{H_1(X_L^\varphi)}$ factors through
\[
\begin{diagram}\dgARROWLENGTH=2.5em
      \node{H_1(X_L^\varphi)}\arrow{s}\arrow[2]{e,t}{ \chi|_{H_1(X_L^\varphi)}}\node[2]{\Z/q^k} 
        \\
      \node{H_1(X_L^\varphi,Y_a^\varphi)} \arrow{nee,b}{\delta} 
    \end{diagram}
\]
and $\delta$ vanishes on $P$. It remains to prove that $\tau(L,\chi)=0$.

We have the following facts and remarks. 

\begin{enumerate}
\item\label{item:character-extension} From the arguments of  \cite[Propositions 2.10, 2.12]{Friedl-Powell:2011-1} ($W_0$ and $W$ play the role of $E_C$ and $W_C$, respectively), we have the following: if $\delta$ vanishes on $P$, then there exist an integer $l\geq k$ and a character $H_1(W^\varphi) \to \mathbb{Z}/q^l$, denoted by $\chi$ as an abuse of notation, which fits into the following diagram:
\[
\begin{diagram}
\node{\pi_1(M_L^\varphi)}\arrow{e}\arrow{s,l}{\operatorname{inc}_*}\node{H_1(M_L^\varphi)}\arrow{e,t}{\chi}\arrow{s,l}{\inc_*}\node{\Z/q^k}\arrow{e,t}{q^{l-k}}\node{\Z/q^l}\\
\node{\pi_1(W^\varphi)}\arrow{e}\node{H_1(W^\varphi)}\arrow{nee,b,..}{\chi}
\end{diagram}
\]
\item\label{item:obstruction} Let $H_1(M_L)=\H'$ and $\phi'\colon \pi_1(M_L)\to H_1(M_L)$ be the Hurewicz homomorphism. Define $\phi\colon \pi_1(M_L^\varphi)\to \H'$ be the restriction of $\phi'$ to the subgroup $\pi_1(M_L^\varphi)$.  Let $\H=\Im{\phi}$. Since $H_1(M_L)\cong H_1(W)$, $\phi'$ extends to $\pi_1(W)$. Therefore, we use $\phi'\colon \pi_1(W)\to \H'$ and its restriction $\phi\colon \pi_1(W^\varphi)\to \H$ as an abuse of notation. Note that $\H'$ is isomorphic to $\Z^3$ and $\H$ is also isomorphic to $\Z^3$ as a finite index subgroup of $\H$.
\item\label{item:twisted-coefficient} By (\ref{item:character-extension}) and (\ref{item:obstruction}), we have $\chi\times \phi \colon \pi_1(W^\varphi)\to \Z/q^l\times \H$. If we write $\K=\C(\H)$, then $H_*(M_L^\varphi;\K), H_*(W_{\vphantom{L}}^\varphi;\K)$, and $H_*(W_{\vphantom{L}}^\varphi,M_L^\varphi;\K)$ can be defined from
\begin{equation*}
\Z[\pi_1(W^\varphi)]\xrightarrow{\chi\times \phi} \Z[\Z/q^l\times \H]=\Z[\Z/q^l][\H]\to \Q(\xi_{q^l})(\H)\to \mathbb{C}(\H)=\mathcal{K}.\end{equation*}

\item\label{item:trivial-character}By \cite[Lemma 3.4]{Friedl-Powell:2011-1},  there is a 4-manifold $W_\chi$ bounded by $2q^l$ copies of a 3-torus $T^3$, which is over $\mathbb{Z}/q^l\times \mathcal{H}$ as follows:
   \[\begin{diagram}\dgARROWLENGTH=1.9em
      \node{\bigsqcup^{q^l}\,T^3}\arrow{s,J}\arrow{see,t}{\chi\times \psi}
        \\
      \node{W_\chi} \arrow[2]{e,t}{\chi\times \psi} 
      \node[2]{K(\Z/q^l\times\H,1)} 
      \\
      \node{\bigsqcup^{q^l}\,T^3} \arrow{n,L}
      \arrow{nee,b}{\operatorname{tr}\times \psi}
    \end{diagram}
    \]Here, $\operatorname{tr}$ denotes the trivial character $\pi_1(T^3)\to \Z/q^l$ and $\psi\colon \pi_1(T^3)\cong \H$. Furthermore, the intersection forms of $W_\chi$ over $\Q$-coefficient and $\K$-coefficients are Witt-trivial.
\end{enumerate}
We can attach $\chi\times \phi\colon W^\varphi \to \Z/q^l\times \H$ and $W_\chi$ in (\ref{item:trivial-character}) along $\chi\times \psi\colon T^3\to \Z/q^l \times \H$ to obtain the cobordism $(W^\varphi\,\cup\,W_\chi,\chi\times \phi\,\cup\,\chi\times \psi)$ over $\Z/q^l\times \H$ between $(M_L^\varphi,\chi\times \phi)$ and $-(T^3,\operatorname{tr}\times \psi)$. From  Definition \ref{definition:Friedl-Powell-invariant}, 
\[\tau(L,\chi)=(\lambda_\K(W^\varphi\,\cup\,W_\chi)-\K\otimes \lambda_\Q(W^\varphi\,\cup\,W_\chi))\otimes1\in L^0(\K)\mathbin{\mathop{\otimes}_{\Z}}\Z[1/q].\]

By (\ref{item:trivial-character}), $[\lambda_\Q(W_\chi)]=0\in L^0(\Q)$ and $[\lambda_\K(W_\chi)]=0\in L^0(\K)$. In the following two claims we will prove that $[\lambda_\Q(W^\varphi)]=0\in L^0(\Q)$ and $[\lambda_{\K}(W^\varphi)]=0\in L^0(\K)$. By Novikov additivity, these will complete the proof of Theorem A.

\begin{claim1} $[\lambda_{\Q}(W^\varphi)]=0\in L^0(\Q)$.
\end{claim1}
\begin{proof}[Proof of Claim 1]
 Applying relative Mayer-Vietoris (see \cite[page 152]{Hatcher:2002}), $H_i(W_0,X_J)\cong H_i(W,M_J)$ for all $i$ and $J=L$ or $H$. (The other terms in the long exact sequence vanish because $H_*(X_H\times I, X_H)=0$.) Similarly, $H_i(W_0^\varphi,X_J^\varphi)\cong H_i(W_{\vphantom{J}}^\varphi,M_J^\varphi)$.

For brevity, let $A=\Z/p^i\oplus \Z/p^j$ and write $A=\{g_1,\ldots,g_t\}$. Since $(W_0;X_L,X_H)$ is a 1.5-solvable cobordism,  by Proposition \ref{proposition1} (1) and (2), 
\[\beta_2(W^\varphi,M_L^\varphi)=\beta_2(W_0^\varphi,X_L^\varphi)=|A|\cdot \beta_2(W_0,X_L)=2rt\] and $\inc_*\colon H_2(W_0^\varphi;\Q)\to H_2(W_0^\varphi,X_J^\varphi;\Q)$ is surjective for $J=L,H$. Since $H_2(W_0^\varphi,X_J^\varphi;\Q)\cong H_2(W_{\vphantom{J}}^\varphi,M_J^\varphi;\Q)$, $\inc_*\colon H_2(W_{\vphantom{J}}^\varphi;\Q)\to H_2(W_{\vphantom{J}}^\varphi,M_J^\varphi;\Q)$ is surjective, too.   Applying Proposition~\ref{proposition1}~(3), $H_i(W_0^\varphi,M_J^\varphi;\Q)=0$ for $i=0,1$. From the homology long exact sequence of a pair $(W_{\vphantom{J}}^\varphi,M_J^\varphi)$, this proves that $(W_{\vphantom{J}}^\varphi,M_L^\varphi,M_H^\varphi)$ is an $H_1$-cobordism over $\Q$-coefficients.

Recall $\d W=M_L\,\sqcup\,-M_H$. For $X=\d W_{\vphantom{L}}^\varphi, M_L^\varphi$, and $M_H^\varphi$, let 
\[I_X=\Im (\operatorname{inc}_*\colon H_2(X;\Q)\to H_2(W^\varphi;\Q)).\] For $J=L,H$, using the homology long exact sequences of pairs,  
\[H_2(W_{\vphantom{J}}^\varphi;\Q)/I_{\d W^\varphi}\cong H_2(W_{\vphantom{J}}^\varphi;\Q)/I_{M_J^\varphi}\cong H_2(W_{\vphantom{J}}^\varphi,M_J^\varphi;\Q)\]
whose rank is $2rt$. (Similar argument was used in the proof of \cite[Proposition 2.6]{Cochran-Kim:2004-1}.) We remark that to prove the last isomorphism, we used the fact that $\operatorname{inc}_*\colon H_1(M_J^\varphi;\Q)\to H_1(W_{\vphantom{J}}^\varphi;\Q)$ is an isomorphism for $J=L,H$.

Let  $l_1,\ldots,l_r,d_1,\ldots,d_r$ be (generators of) $2$-lagrangian and $1$-duals in $H_2(W^\varphi;\Z)$. From the (right) group action of $A$ on $H_2(W^\varphi;\Z)$, we define
\[ l_{ij}=l_i\cdot g_j\textrm{~and~} d_{kl}=d_k\cdot g_l \textrm{~for~} 1\leq i,k\leq r\textrm{~and~}1\leq j,l\leq t.\]

The intersection pairing $\lambda_\Q(W^\varphi)\colon H_2(W^\varphi;\Q)/I_{\d W^\varphi}\times H_2(W^\varphi;\Q)/I_{\d W^\varphi}\to \Q$ with respect to (the image of) $\{l_{ij},d_{kl}\}$ is \[\begin{pmatrix}
        0&I_{rt\times rt}\\
        I_{rt\times rt}&X
     \end{pmatrix}\]
     because $l_i\cdot d_k$ is the Kronecker delta $\delta_{ik}$.

Let $L(\Q)\subset H_2(W^\varphi;\Q)/I_{\d W^\varphi}$ be the $\Q$-span of the image of $l_{ij}\otimes 1_{\Q}$. Then, $\lambda_\Q(W^\varphi)$ vanishes on $L(\Q)\times L(\Q)$ and $\dim_\Q L(\Q)=\frac{1}{2}\dim_\Q (H_2(W^\varphi;\Q)/I_{\d W^\varphi})=rt$. So, $[\lambda_\Q(W^\varphi)]=0\in L^0(\Q)$. 
\end{proof}
By \cite[Lemma 3.2]{Friedl-Powell:2011-1}, $H_*(M_J^\varphi;\K)=0$ for $J=H$ or $L$. Therefore, the twisted intersection form
\[\lambda_\K(W^\varphi)\colon H_2(W^\varphi;\K)\times H_2(W^\varphi;\K)\to \K\] is non-singular. 

\begin{claim2}$[\lambda_\K(W^\varphi)]=0\in L^0(\K)$.
\end{claim2}
\begin{proof}[Proof of Claim 2]
Let
$ \alpha\colon \pi_1(W^\varphi)\xrightarrow{\chi} \Z/q^l\hookrightarrow \mathbb{C}^\times=\GL(1,\mathbb{C})$ and $\alpha'\colon \pi_1(W)\to \GL(t,\mathbb{C})$ be the induced representation of $\alpha$. Recall $\phi'\colon \pi_1(W)\to \H'$ and $\phi\colon\pi_1(W^\varphi)\to \H$ in (\ref{item:obstruction}). Define $\Gamma:=\Im(\alpha\times \phi)$. There is a corresponding cover $(W_{\vphantom{L}}^\Gamma,M^\Gamma_L)\to (W^\varphi,M_L^\varphi)$ where $\pi_1(W^\Gamma)=\Ker(\alpha\times\phi)$. Recall $W^\varphi\to W$ is $\Z/p^i\oplus \Z/p^j$-cover and $\alpha\times \phi\colon \pi_1(W^\varphi)\to \C^\times\times \H$.  Since $\Z/p^i\oplus \Z/p^j$, $\C^\times$, and $\H$ are abelian, 
\[\pi_1(W)^{(2)}\leq \pi_1(W^\varphi)^{(1)}\leq \Ker(\alpha\times \phi)=\pi_1(W^\Gamma)\]

Equivalently, there is a sequence of coverings:
\[\begin{diagram}
\node{W^{(2)}}\arrow{e}\node{W^\Gamma}\arrow{e}\node{W^\varphi}\arrow{e}\node{W}
\end{diagram}
\]

Since $\Z/q^l\hookrightarrow \C^\times$ is injective,  $\Ker \alpha=\Ker \chi$, where $\alpha\colon \pi_1(W^\varphi)\xrightarrow{\chi}\Z/q^l\hookrightarrow \C^\times$. From this, $\Gamma\overset{\textrm{def}}=\Im(\alpha\times\phi)=\Im(\chi\times\phi)$.
In particular, the ring homomorphism $\Z[\pi_1(W^\varphi)]\to \K$ in (\ref{item:twisted-coefficient}) factors through $\Z\Gamma$ and 
\[ C_*(W^\varphi;\K)\overset{\textrm{def}}=\K\mathbin{\mathop{\otimes}_{\Z[\pi_1(W^\varphi)]}}C_*(W^\varphi;\Z[\pi_1 W^\varphi])=\K\mathbin{\mathop{\otimes}_{\Z\Gamma}}C_*(W^{\Gamma};\Z).
\]
Choose 2-cycles $\{\tilde{l_1},\ldots,\tilde{l_r}\}\subset C_2(W^\Gamma;\Z)$ which represent the image of (generators of) 2-lagrangian under the map induced by $W_0^{(2)}\to W^{(2)}\to W^\Gamma$. The covering map $W^\Gamma\to W^\varphi$ induces a surjection $\Cov(W^\Gamma|W)\to \Cov(W^\varphi|W)=\{g_1,\ldots,g_t\}$. Choose a lift $\tilde{g_j}\in \Cov(W^\Gamma|W)$ of $g_j$ for each $j=1,\ldots,t$. From the right action of $\Cov(W^{\Gamma}|W)$ on $C_2(W^{\Gamma};\Z)$, define 
\[ \tilde{l_{ij}}=\tilde{l_i}\cdot \tilde{g_j}\textrm{~for~}1\leq i\leq r\textrm{~and~}1\leq j\leq t.\]

Let
\[ L(\K)\subset H_2(W^\varphi;\K)=H_2(\K\mathbin{\mathop{\otimes}_{\Z\Gamma}}C_*(W^{\Gamma};\Z)),\] be the $\K$-span of $\{[1_\K \otimes \tilde{l_{ij}}]~|~1\leq i\leq r, 1\leq j\leq t\}$ in $H_2(W^\varphi;\K)$. We remark that $L(\K)$ does not depend on the choice of $\tilde{g_i}$. We claim that $L(\K)$ is a lagrangian for the non-singular twisted intersection form $\lambda_\K(W^\varphi)$.

First, we prove $\lambda_\K$ vanishes on $L(\K)\times L(\K)$. Since $\lambda_\K$ is $\K$-sesquilinear, the following is enough:
\[ \lambda_\K ([1_\K\mathbin{\mathop{\otimes}_{\Z\Gamma}} \tilde{l_{ik}}],[1_\K\mathbin{\mathop{\otimes}_{\Z\Gamma}}\tilde{l_{jl}}])=\sum_{g\in \Cov(W^{\Gamma}|W)} \lambda_{W^{\Gamma}}(\tilde{l_i},\tilde{l_j})g_l^{\vphantom{-1}}gg_k^{-1}=0.\]

Now, we prove $\dim_\K L(\K)=\frac{1}{2}\dim_\K H_2(W^\varphi;\K)$. Recall that $H_*(M_L;\K)=0$ by \cite[Lemma 3.2]{Friedl-Powell:2011-1}. Therefore, $\inc_*\colon H_2(W_{\vphantom{L}}^\varphi;\K)\to H_2(W_{\vphantom{L}}^\varphi,M_L^\varphi;\K)$ is an isomorphism. Now, for simplicity, we abuse notation by regarding $\tilde{l_{ij}}$ as an element in $C_2(W_{\vphantom{L}}^\Gamma,M_L^\Gamma;\Z)$ and $L(\K)$ as a subspace of $H_2(W_{\vphantom{L}}^\varphi,M_L^\varphi;\K)$.

Recall that $\{l_1,\ldots, l_r\}$ is the chosen generators of 2-lagrangian in $H_2(W^\varphi;\Z)$. Since the covering  $W^{\Gamma}\to W$ sends $\tilde{g_j}$ to $1$, the image of $\{[\tilde{l_{ij}}]\in H_2(W^\Gamma;\Z)~|~1\leq i\leq r, 1\leq j\leq t\}$ in $H_2(W,M_L;\Z)$ (via covering induced map) is exactly $\{\pi(l_{1}),\ldots,\pi(l_r)\}$ where $\pi\colon H_2(W^\varphi)\to H_2(W)\to H_2(W,M_L)$.  

Since $(W_0;X_L,X_H)$ is a 1.5-solvable cobordism with $\beta_2(W_0,X_L)=2r$, $H_2(W,M_L)\cong H_2(W_0,X_L)$ is a free abelian group of rank $2r$. Let 
\[L(\Z/q)\subset H_2(W,M_L;\Z/q)\cong (\Z/q)^{2r}\]
be the $\Z/q$-span of $\{\pi(l_i)\otimes_\Z1_{\Z/q}\}_{i=1}^r$.  By the definition of 2-lagrangian, $\{\pi(l_1),\ldots,\pi(l_r)\}$ generates a rank $r$-summand of $H_2(W,M_L)\cong \Z^{2r}$. In particular, from the universal coefficient theorem, $\dim_{\Z/q}L(\Z/q)=r$.

To apply Theorem \ref{finalproposition}, we fit our notations with those used in Section \ref{section:Casson-Gordon-type-representations}. Define $A=\Z/p^i\oplus \Z/p^j$, $G=\pi_1(W)$, $K=\pi_1(W^\varphi)$, $C_*=C_*(W,M_L;\Z [\pi_1(W)])$, $Q=\mathbb{C}$, $Q(\H)=\K$, $d=1$, $\alpha\times \phi\colon \pi_1(W^\varphi)\to \C^\times \times \H$, and $\alpha'\times \phi'\colon \pi_1(W)\to \GL(t,\C)\times \H'$. (As a $\Z K$-module, $C_*$ is isomorphic to $C_*(W_{\vphantom{L}}^\varphi,M_L^\varphi;\Z[\pi_1(W^\varphi)])$.) We remark that we assumed in Section \ref{section:Casson-Gordon-type-representations} that $\alpha|_{\Ker\phi}$ factors through a $q$-group for some prime $q$. This is automatically satisfied for $\alpha\colon \pi_1(W^\varphi)\xrightarrow{\chi}\Z/q^l\hookrightarrow \C^\times$.

With these notations, apply Theorem \ref{finalproposition} for the case $I=\emptyset$ (that is, $M=\overline{M}=0$) and $n=0,1$ to obtain 
\[ \dim_\K H_n(W_{\vphantom{J}}^\varphi, M_J^\varphi;\K)\leq \dim_{\Z/q}H_n(W,M_J;\Z/q)=0\]
for $n=0,1$ and $J=L$ or $H$. By duality and universal coefficient spectral sequence, $H_i(W_{\vphantom{L}}^\varphi,M_L^\varphi;\K)=0$ for $i=3,4$. 
From this,  \[ \dim_\K H_2(W_{\vphantom{L}}^\varphi,M_L^\varphi;\K)=\chi^\K(W_{\vphantom{L}}^\varphi,M_L^\varphi)=\chi^\Q(W_{\vphantom{L}}^\varphi,M_L^\varphi)=2rt.\]
The last equality is from $\beta_2(W_{\vphantom{L}}^\varphi,M_L^\varphi)=2rt$ and $(W_{\vphantom{L}}^\varphi;M_L^\varphi,M_H^\varphi)$ is an $H_1$-cobordism over $\Q$-coefficient. These are proved in the proof of Claim 1.

Now, we apply Theorem \ref{finalproposition} for the case $n=2$, $I=\{i~|~1\leq i\leq r\}$ and $x_{i}$ is a 2-cycle in $C_*$ such that
\[[1_\K \mathbin{\mathop{\otimes}_{\Z \Gamma}} \tilde{l_i}]=[1_\K \mathbin{\mathop{\otimes}_{\Z K}} x_i]\in H_2(\K\mathbin{\mathop{\otimes}_{\Z K}}C_*)=H_2(W_{\vphantom{L}}^\varphi,M_L^\varphi;\K)\textrm{~for~}i=1,\ldots,r. 
\]
Recall $\tilde{l_{ij}}=\tilde{l_i}\cdot \tilde{g_j}$, $\tilde{g_j}\in \Cov(W^\Gamma|W)$ is a lifting of $g_j\in \Cov(W^\varphi|W)$. Since $\Cov(W^\varphi|W)$ can be identified with the set of cosets of $K$ in $G$, by the definition in Theorem \ref{finalproposition}, 
\[M= \textrm{the~}\K\textrm{-span of~}\{[1_\K\mathbin{\mathop{\otimes}_{\Z \Gamma}} \tilde{l_{ij}}]~|~1\leq i\leq r, 1\leq j\leq t\}=L(\K).\] 
Similarly, by the definition in Theorem \ref{finalproposition}, $\overline{M}$ is the $\Z/q$-span of $\{[1_{\Z/q}\otimes_{\Z G}x_i]\}_{i=1}^r$. Since $\{[1_{\Z/q}\otimes_{\Z G}x_i]\}_{i=1}^r=\{1_{\Z/q}\otimes_\Z \pi(l_i)\}_{i=1}^r$, 
\[\overline{M}=\textrm{the~}\Z/q\textrm{-span of~}\{1_{\Z/q}\mathbin{\mathop{\otimes}_{\Z}}  \pi(l_i)~|~ 1\leq i\leq r\}=L(\Z/q).\]
From the conclusion of Theorem \ref{finalproposition} for the above case, we have the following inequality
\begin{equation*}\dim_{\K} H_2(W_{\vphantom{L}}^\varphi,M_L^\varphi;\K)-\dim_{\K} L(\K)\leq  t\cdot(\dim_{\Z/q} H_2(W,M_L;\Z/q)-\dim_{\Z/q}L(\Z/q)).\end{equation*}That is,
\begin{eqnarray*}\dim_{\K}L(\K)&\geq& 
\dim_{\K}H_2(W,M_L;\K)-t\cdot(\dim_{\Z/q}H_2(W,M_L;\Z/q)-\dim_{\Z/q}L(\Z/q))\\
&=&2rt-rt+rt=rt.
\end{eqnarray*}
On the other hand, $\dim_{\K}L(\K)\leq rt$ because $L(\K)$ is the $\K$-span of $rt$ elements. So, $\dim_{\K}L(\K)=rt=\frac{1}{2}\dim_\K H_2(W^\varphi;\K)$ and $L(\K)$ is a lagrangian of $\lambda_\K(W^\varphi)$. That is, $[\lambda_\K(W^\varphi)]=0\in L^0(\K)$. 
\end{proof}
\section{Solvable cobordism and abelian invariants of links}\label{abelianinvariants}

In this section, we study the abelian invariants of links (studied in \cite{Kawauchi:1978-1} and \cite{Hillman:2012-1}) in the context of Whitney tower/grope concordance using $h$-solvable cobordism. Throughout this section, $\mu$ is the fixed natural number. Denote $\Z[t_1^\pm,\ldots,t_\mu^\pm]$ by $\Lambda_\mu$. The ring $\Lambda_\mu$ is endowed with the involution $-\colon t_i\mapsto t_i^{-1}$. Let $S$ be the multiplicative set generated by $\{t_1-1,\ldots,t_\mu-1\}$. Denote the localization of $\Lambda_\mu$ with respect to $S$ by $\Lambda_{\mu S}$. Let $\K$ be the quotient field of $\Lambda_\mu$.

\subsection{Blanchfield form of $\mu$-component links}\label{subsection:Blanchfieldform}
Let $L$ be a $\mu$-component link and $X_L$ be the link exterior of $L$. Let $R$ be a unique factorization domain with an involution $-$ and quotient field $K$ (our case is $R=\Lambda_{\mu S}, K=\K$). We recall the definition of the Witt group $W(K,R,-)$.

A \emph{linking pairing over $R$} is a $R$-module $M$ with a sesquilinear pairing 
\[ b\colon M\times M \to K/R\] such that for all $x,y,z\in M$ and $r\in R$
\begin{enumerate}
\item $b(x,y+z)=b(x,y)+b(x,z)$
\item $b(rx,y)=rb(x,y)=b(x,\overline{r}y)$
\item $b(x,y)=\overline{b(y,x)}$
\end{enumerate}
(Here, the involution $-$ on $K/R$ is induced from the involution on $R$.) We denote it by $(M,b)$ or just $b$ when $M$ is clearly understood. 
A linking pairing $(M,b)$ is \emph{primitive} (\emph{non-singular}) if the adjoint of $b$, 
\[\operatorname{Ad}(b)\colon M\to \Hom_R(M,K/R)\] is an injection (an $R$-module isomorphism), respectively.
The sum of linking pairings $(M,b)$ and $(M',b')$ is $(M\oplus M',b\oplus b')$. A pairing $(M,b)$ is \emph{neutral} if there is a submodule $N$ of $M$ such that 
\[ N=N^\perp =\{m\in M~|~b(n,m)=0~\forall~n\in N\}.\]
Two pairings $(M,b)$ and $(M',b')$ are \emph{Witt equivalent} if there are neutral pairings $(N,c)$ and $(N',c')$ such that $(M,b)\oplus (N,c)\cong (M',b')\oplus (N',c')$. Then, the set of Witt equivalence classes of linking pairings over $R$ with an involution $-$ is an abelian group, denoted by $W(K,R,-)$.

For a $R$-module $M$, following \cite[Chapter 3]{Hillman:2012-1}, we define the $R$-torsion submodule of $M$, 
\[ tM=\{m\in M~|~rm=0~\textrm{for~some~}r\neq 0\in R\}=\Ker (M\to M\otimes_R K),\] the maximal pseudonull submodule of $M$, 
\[zM=\Ker (tM\to \Ext^1_R(\Ext^1_R(tM,R),R)),\] 
and 
\[\hat{t}M=tM/zM.\] 
Note that a $R$-module $M$ is called \emph{pseudonull} if $M_{\mathfrak{p}}=0$ for every height 1 prime ideal $\mathfrak{p}$ of $R$. 

From the Alexander duality, the Hurewicz map becomes $\pi_1(X_L)\to H_1(X_L)=\Z^\mu$. We have the following exact sequence 
\[ H_1(\partial X_L;\Lambda_\mu)\to H_1(X_L;\Lambda_\mu)\to H_1(X_L,\partial X_L;\Lambda_\mu)\to H_0(\partial X_L;\Lambda_\mu)\]
whose extermal terms are $\prod\limits_{i=1}^\mu (t_i-1)$-torsion (in particular, $S$-torsion) because $\Z^\mu$-cover of $\partial X_L$ is a disjoint union of $S^1\times \R$ or $\R\times \R$. From this observation, by localizing the above sequence with respect to $S$, we obtain $H_1(X_L;\Lambda_{\mu S})\cong H_1(X_L,\partial X_L;\Lambda_{\mu S})$. It follows from the (localized) Blanchfield duality \cite{Blanchfield:1957-1} (as in \cite[page 36]{Hillman:2012-1}) that we have the following primitive linking pairing :
\[ b_L\colon \hat{t}H_1(X_L;\Lambda_{\mu S})\times \hat{t}H_1(X_L;\Lambda_{\mu S})\to \K/\Lambda_{\mu S}.\]
Here, to define $b_L$, we need the fact that $\K/\Lambda_{\mu S}$ contains no nontrivial pseudonull submodule, \cite[Theorem 3.9 (2)]{Hillman:2012-1}.

In this setting, Hillman \cite[Theorem 2.4]{Hillman:2012-1} proved that $[b_L]\in W(\K,\Lambda_{\mu S},-)$ is a concordance invariant of $L$. Here is our theorem which generalizes \cite[Theorem 2.4]{Hillman:2012-1}.

 \begin{theoremB}\label{blanchfieldform}Suppose $L_0$ and $L_1$ are $\mu$-component links. If two link exteriors $X_{L_0}$ and $X_{L_1}$ are 1-solvable cobordant, then $[b_{L_0}]=[b_{L_1}]\in W(\K,\Lambda_{\mu S},-)$. In particular, the conclusion holds if $L_0$ and $L_1$ are height 3 Whitney tower/grope concordant.\end{theoremB}
 \begin{proof} Let $W$ be a 1-solvable cobordism between $X_{L_0}$ and $X_{L_1}$. Note that
 \[ \partial W=X_{L_0}\cup \mu (S^1\times S^1\times I)\cup -X_{L_1}\] 
 and $\Z^\mu=H_1(X_{L_i})\xrightarrow{\operatorname{inc}_*} H_1(W)$ is an isomorphism for $i=0,1$. 
 
 By the ($\Lambda_{\mu S}$-coefficient) Mayer-Vietoris sequence of the triple $(\partial W, X_{L_0},X_{L_1})$, \[H_1(\partial W;\Lambda_{\mu S})\cong H_1(X_{L_0};\Lambda_{\mu S})\oplus H_1(X_{L_1};\Lambda_{\mu S}),\]
since $ H_i(\mu (S^1\times S^1\times I);\Lambda_\mu)$ is $S$-torsion for $i=0,1$. From this, the (localized) Blanchfield form 
 \[ b_{\partial W}\colon \hat{t}H_1(\partial W;\Lambda_{\mu S})\times \hat{t}H_1(\partial W;\Lambda_{\mu S})\to \K/\Lambda_{\mu S}\] is the direct sum $b_{L_0}\oplus (-b_{L_1})$. Therefore, it suffices to find a submodule $Q$ of $\hat{t}H_1(\partial W;\Lambda_{\mu S})$ such that $Q=Q^\perp$.

 By applying \cite[Theorem 4.13]{Cha:2012-1} to $n=1$, $G=\Z^\mu$, $\phi\colon \pi_1(W)\to H_1(W)=\Z^\mu$, and $R=\Z$, we have the following Lemma.
 \begin{lemma}\cite[Theorem 4.13]{Cha:2012-1}\label{abelianlemma} In the above situation, 
\[ tH_2(W,\partial W;\Lambda_\mu)\to tH_1(\partial W;\Lambda_\mu)\to tH_1(W;\Lambda_\mu)\]
 is exact.
 \end{lemma}
 
 Let $I_{\partial W}$ and $I_W$ be $\Lambda_{\mu S}$-coefficient intersection forms of $\partial W$ and $W$, respectively. We have Blanchfield form, 
 \[b_W\colon \hat{t}H_1(W;\Lambda_{\mu S})\times \hat{t}H_2(W,\partial W;\Lambda_{\mu S})\to \K/\Lambda_{\mu S}.\]
 Let $P=\Im(tH_2(W,\partial W;\Lambda_{\mu S})\to \hat{t}H_1(\partial W;\Lambda_{\mu S}))$. Choose relative 2-cycles $Q$ and $R$ in $C_2(W,\partial W;\Lambda_{\mu S})$ representing the classes in $tH_2(W,\partial W;\Lambda_{\mu S})$. Denote the boundaries of $Q$ and $R$ by $q,r\in C_1(\partial W;\Lambda_{\mu S})$, respectively. The corresponding classes $[q],[r]$ in $\hat{t}H_1(\partial W;\Lambda_{\mu S})$ are actually in $P$. There exists $a\in \Lambda_{\mu S}-\{0\}$ such that $aq=\partial u$ for some $u\in C_2(\partial W;\Lambda_{\mu S})$. Then,
 \[b_{\partial W}([q],[r])=a^{-1}I_{\partial W}(u,r)=-a^{-1}I_W(i_*(u),R)=-b_W([\partial Q],[R])\quad  (\operatorname{mod} \Lambda_{\mu S}).\]
(Here, $i_*\colon C_2(\partial W;\Lambda_{\mu S})\to C_2(W;\Lambda_{\mu S}$).)
Note that $[\partial Q]=0\in \hat{t}H_1(W;\Lambda_{\mu S})$. Therefore, 
 \[ b_{\partial W}([q],[r])=-b_W([\partial Q],[R])=0\textrm{~for all~}[q],[r]\in P.\]
This shows that $P\leq P^\perp$. Suppose that $x\in C_1(\partial W;\Lambda_{\mu S})$ represents a torsion class in $tH_1(\partial W;\Lambda_{\mu S})$ and $[x]\in P^\perp$. That is, 
\[ b_{\partial W}([x],[y])=0\textrm{~for all~}y=\partial Y,~[Y]\in \hat{t}H_2(W,\partial W;\Lambda_{\mu S}).\]
So,
\[b_W([x],[Y])=-b_{\partial W}([x],[y])=0\textrm{~for all~}[Y]\in \hat{t}H_2(W,\partial W;\Lambda_{\mu S}).\]
By Blanchfield duality for $(W,\partial W)$, the adjoint of $b_W$, 
\[\operatorname{Ad}(b_W)\colon \hat{t}H_1(W;\Lambda_{\mu S})\to \operatorname{Hom}_{\Lambda_{\mu S}}(\hat{t}H_2(W,\partial W;\Lambda_{\mu S}),K/\Lambda_{\mu S})\] is injective. Therefore, $[x]=0\in \hat{t}H_1(W;\Lambda_{\mu S})$ or $x$ represents a homology class in $zH_1(W;\Lambda_{\mu S})$. This shows that $P^\perp/P$ is a pseudonull $\Lambda_{\mu S}$-module. We claim that $P^{\perp\perp}=P^\perp$. 
The inclusion maps $P\hookrightarrow P^\perp$ and $P^\perp\hookrightarrow \hat{t}H_1(\partial W;\Lambda_{\mu S})$ induce two vertical maps in the following diagram (here, the horizontal map is $\operatorname{Ad}(b_{\partial W})$):
\[
\begin{diagram}
\node{\hat{t}H_1(\partial W;\Lambda_{\mu S})}\arrow{e}\arrow{se,b}{i}\arrow{sse,b}{j\circ i}\node{\Hom_{\Lambda_{\mu S}}(\hat{t}H_1(\partial W;\Lambda_{\mu S}),\K/\Lambda_{\mu S})}\arrow{s}\\
\node[2]{\Hom_{\Lambda_{\mu S}}(P^\perp,\K/\Lambda_{\mu S})}\arrow{s,r}{j}\\
\node[2]{\Hom_{\Lambda_{\mu S}}(P,\K/\Lambda_{\mu S})}
\end{diagram}
\]
By definition, $P^{\perp\perp}=\Ker i$ and $P^{\perp}=\Ker({j\circ i})$. By applying $\Hom_{\Lambda_{\mu S}}(-,\K/\Lambda_{\mu S})$ to $0\to P\to P^\perp\to P^\perp/P\to 0$, we obtain that
\[0\to \Hom_{\Lambda_{\mu S}}(P^\perp/P,\K/\Lambda_{\mu S})\to\Hom_{\Lambda_{\mu S}}(P^\perp,\K/\Lambda_{\mu S})\xrightarrow{j} \Hom_{\Lambda_{\mu S}}(P,\K/\Lambda_{\mu S})\]
is exact. That is, $\Ker j\cong\Hom_{\Lambda_{\mu S}}(P^\perp/P,\K/\Lambda_{\mu S})$.

From the short exact sequence $0\to \Lambda_{\mu S}\to \K\to \K/\Lambda_{\mu S}\to 0$, the following is exact :
\[\Hom_{\Lambda_{\mu S}}(P^\perp/P,\K)\to \Hom_{\Lambda_{\mu S}}(P^\perp/P,\K/\Lambda_{\mu S})\to \Ext^1_{\Lambda_{\mu S}}(P^\perp/P,\Lambda_{\mu S}).\]
Since $P^\perp/P$ is $\Lambda_{\mu S}$-torsion and $\K$ is $\Lambda_{\mu S}$-torsion free, $\Hom_{\Lambda_{\mu S}}(P^\perp/P,\K)=0$. Also, $P^\perp/P$ is a pseudonull  $\Lambda_{\mu S}$-module implies that $\Ext^1_{\Lambda_{\mu S}}(P^\perp/P,\Lambda_{\mu S})=0$. (By Theorem 3.9 of \cite{Hillman:2012-1}, for a unique factorization domain $R$ and $R$-module $M$, $M$ is pseudonull if and only if $\Hom_R(M,R)=0$ and $\Ext^1_R(M,R)=0$.) Therefore, 
\[ \Ker j\cong\Hom_{\Lambda_{\mu S}}(P^\perp/P,\K/\Lambda_{\mu S})=0.\]
From the kernel-cokernel exact sequence, 
\[ 0\to \Ker i\to \Ker (j\circ i)\to \Ker j=0\] is exact. This shows that $P^\perp=\Ker (j\circ i)\cong\Ker i=P^{\perp\perp}$ and completes the proof. \end{proof}
\subsection{Multivariable Alexander polynomial of links}\label{subsection:Alexanderpolynomial}

In this subsection, we prove Theorem C which generalizes \cite[Theorems A, B]{Kawauchi:1978-1} concerning the Fox-Milnor condition for the Alexander polynomial of links. 

First, we recall some definitions of \cite{Kawauchi:1978-1}. Since $\Lambda_\mu$ is N\"{o}therian, for a finitely generated $\Lambda_\mu$-module $M$, we can choose a presentation matrix $P$ of $M$ from an exact sequence $\Lambda_\mu^m\xrightarrow{P}\Lambda_\mu^n\to M\to 0$. Moreover, for all $k$, one can choose a $m\times n$ presentation matrix $P$ with $n>k$ and $m\geq n-k$.  In this situation, define \emph{the $k$-th Alexander polynomial of $M$}, denoted by $\Delta_k(M)$, to be the greatest common divisor of the size $(n-k)\times (n-k)$ minors of $P$. (It is well-known that $\Delta_k(M)$ is well-defined up to a unit of $\Lambda_\mu$ which is proved in \cite{Crowell-Fox:1977-1}.)

\begin{remark}\label{easylemma}
(1) From \cite[Theorem 4.10]{Blanchfield:1957-1}, if $d=\rank_{\Lambda_\mu} M$, then $\Delta_d(M)=\Delta_0(tM)$. \\
(2) From \cite[Lemma 2.4]{Kawauchi:1978-1}, if $0\to M'\to M\to M''\to 0$ is a short exact sequence of $\Lambda_\mu$-torsion modules, then $\Delta_0(M)=\Delta_0(M')\Delta_0(M'')$.\\
(3) By definition, for a nonzero $\Lambda_\mu$-torsion module $M$, $\Delta_0(M)\neq 0$. 
\end{remark}

Recall that $L$ is a $\mu$-component link in $S^3$ and the meridian map is $\pi_1(X_L)\to \Z^\mu$. We define the torsion Alexander polynomial of $L$ by $\Delta_L^T:=\Delta_0(tH_1(X_L;\Lambda_\mu))$. Now we state our theorem.
\begin{theoremC} Suppose $L_0$ and $L_1$ are $\mu$-component links. If two link exteriors $X_{L_0}$ and $X_{L_1}$ are 1-solvable cobordant, then 
\begin{enumerate}
\item $\rank_{\Lambda_\mu}H_1(X_{L_0};\Lambda_\mu)=\rank_{\Lambda_\mu}H_1(X_{L_1};\Lambda_\mu)$ and
\item $\Delta_{L_0}^Tf_0 \overline{f_0}\overset{\cdot}=\Delta_{L_1}^T f_1\overline{f_1}$ for some $f_i(t_1,\ldots,t_\mu)\in \Lambda_\mu$, $i=0,1$ with $|f_i(1,\ldots,1)|=1$. 
\end{enumerate}
In particular, the conclusion holds if $L_0$ and $L_1$ are height 3 Whitney tower/grope concordant.
\end{theoremC}
To prove  Theorem C, we need to prove the following generalization of \cite[Lemma~2.1]{Kawauchi:1978-1}. 
\begin{lemma}\label{lemma2.1kawauchi}Let $X$ be a finite connected CW-complex with an epimorphism 
\[\gamma\colon \pi_1(X)\to \Z^\mu\] Let $X_0$ be a subcomplex  of $X$. For some fixed $k$, if $H_k(X,X_0;\Z)=\Z^l$ and 
\[\rank_{\Lambda_\mu} H_k(X,X_0;\Lambda_\mu)=l\] then the $l$-th Alexander polynomial $A=\Delta_l(H_k(X,X_0;\Lambda_\mu))=\Delta_0(tH_k(X,X_0;\Lambda_\mu))$ satisfies $|A(1,\ldots,1)|=1$.
\end{lemma}
\begin{remark0}For the $l=0$ case  (\cite[Lemma 2.1]{Kawauchi:1978-1}), we only need to assume $H_k(X,X_0;\Z)=0$ because from our proof, we can deduce 
\[\rank_{\Lambda_\mu} H_k(X,X_0;\Lambda_\mu) \leq 0.\]
In this sense, Lemma \ref{lemma2.1kawauchi} is a generalization of \cite[Lemma 2.1]{Kawauchi:1978-1}.
\end{remark0}
\begin{proof}[Proof of Lemma \ref{lemma2.1kawauchi}] Since $X_0$ is a subcomplex of $X$,  for all $q$, we fix a basis for the $q$-th (cellular) chain complex $C_q(X,X_0;\Z)\cong \Z^{s_q}$. By lifting each element in the chosen bases, we also fix a $\Lambda_\mu$-basis for the $C_q(X,X_0;\Lambda_\mu)$ for all $q$. With these chosen bases, we can write $\partial_q\colon C_q(X,X_0;\Lambda_\mu)\to C_{q-1}(X,X_0;\Lambda_\mu)$ as a matrix $(\alpha_{ij}^q)$, $\alpha_{ij}^q\in \Lambda_\mu$. 

With respect to the chosen basis of $C_*(X,X_0;\Z)$, $\partial_q\colon C_q(X,X_0;\Z)\to C_{q-1}(X,X_0;\Z)$ is represented by the integral matrix $(\alpha_{ij}^q(1,\ldots,1))$. Let $\tilde{r_q}=\operatorname{rank}(\alpha_{ij}^q)$, $r_q=\operatorname{rank}(\alpha_{ij}^q(1,\ldots,1))$. Then, $r_q\leq \tilde{r_q}$. Since $H_k(X,X_0;\Z)=\Z^l$, 
\[ l=\rank_\Z \Ker \partial_k-\rank_\Z \Im\partial_{k+1}=s_k-r_k-r_{k+1}.\]
Similarly, from $\rank_{\Lambda_\mu} H_k(X,X_0;\Lambda_\mu)=l$,
\[ l=s_k-\tilde{r_k}-\tilde{r_{k+1}}\]
Since $r_q\leq \tilde{r_q}$ for all $q$,
\[ l=s_k-\tilde{r_k}-\tilde{r_{k+1}}\leq s_k-r_k-r_{k+1}=l\] which implies that $r_k=\tilde{r_k}, r_{k+1}=\tilde{r_{k+1}}$.

Since $\Coker\partial_{k+1}=C_k(X,X_0;\Lambda_\mu)/\Im \partial_{k+1}$ and $\Im \partial_k\cong C_k(X,X_0;\Lambda_\mu)/\Ker \partial_k$, we have the short exact sequence
\[0\to H_k(X,X_0;\Lambda_\mu)\to \Coker \partial_{k+1}\to \Im \partial_i\to 0\]
As a submodule of a free module, $\Im\partial_k$ is a $\Lambda_\mu$-torsion free module of rank $\tilde{r_k}=r_k$. Then, $tH_k(X,X_0;\Lambda_\mu)=t\Coker\partial_{k+1}$ and $\dim_\K\Coker\partial_{k+1}\otimes_{\Lambda_\mu}\K=l+r_k$.
\[\Delta_l(H_k(X,X_0;\Lambda_\mu))=\Delta_0(tH_k(X,X_0;\Lambda_\mu))=\Delta_0(t\Coker\partial_{k+1})=\Delta_{l+r_k}(\Coker\partial_{k+1})\] 
The first and last inequality follows from Remark \ref{easylemma} (1).
Similarly, we have the short exact sequence, 
\[ 0\to H_k(X,X_0;\Z)\to \Coker\partial_{k+1}^\Z\to \Im\partial_k^\Z\to 0.\]
(Here, to avoid the confusion, we denote the differential on $C_*(X,X_0;\Z)$ by $\partial_*^\Z$.)
Since $\Z$ is a PID, every submodule of finitely generated free $\Z$-module is free. So, $\Im\partial_k^\Z$ is isomorphic to $\Z^{r_k}$. Therefore, 
\[ \Coker\partial_{k+1}^\Z=H_k(X,X_0;\Z)\oplus \Z^{r_k}=\Z^{l+r_k}\]
(Here, we used the assumption that $H_k(X,X_0;\Z)=\Z^l$.)
Note that the matrices $(\alpha_{ij}^{k+1})$ and $(\alpha_{ij}^{k+1}(1,\ldots,1))$ are presentation matrices of $\Coker\partial_{k+1}$ and  $\Coker\partial_{k+1}^\Z$, respectively. Therefore, 
\[|\Delta_l(H_k(X,X_0;\Lambda_\mu))(1,\ldots,1)|=|\Delta_{l+r_k}(\Coker\partial_{k+1})(1,\ldots,1)|=1.\]  This completes the proof.
\end{proof}
\begin{proof}[Proof of Theorem C]Let $W$ be a 1-solvable cobordism between $X_{L_0}$ and $X_{L_1}$. In particular, the inclusion induces $\Z^\mu=H_1(X_{L_0})\cong H_1(W)$ and $H_1(W,X_{L_0})=H_1(W,X_{L_1})=0$. By Poincar\'{e} duality and universal coefficient theorem, 
\[H_2(W,X_{L_0})\cong H^2(W,X_{L_1})\cong \Hom_\Z (H_2(W,X_{L_1}),\Z)=\Z^{2r}\] (Since $W$ is a 1-solvable cobordism between $X_{L_0}$ and $X_{L_1}$, $\rank_\Z H_2(W,X_{L_1})$ is even.) Let $C_*=C_*(W,X_{L_0};\Lambda_\mu)$. Then, 
\[ H_i(C_*\mathbin{\mathop{\otimes}_{\Lambda_\mu}} \Z)=H_i(W,X_{L_0};\Z)=0\textrm{~for~}i=0,1.\] Since $\Lambda_\mu=\Z[\Z^\mu]$ and $\Z^\mu$ is a poly-torsion-free-abelian-group, by \cite[Proposition 2.10]{Cochran-Orr-Teichner:2003-1}, 
\[ H_i(C_*\mathbin{\mathop{\otimes}_{\Lambda_\mu}}\K)=H_i(W,X_{L_0};\Lambda_\mu)\mathbin{\mathop{\otimes}_{\Lambda_\mu}}\K=0\textrm{~for~}i=0,1.\]
Similarly, $H_i(W,X_{L_1};\Lambda_\mu)\otimes_{\Lambda_\mu}\K=0$ for $i=0,1$. From duality and universal coefficient spectral sequence, $H_i(W,X_{L_0};\Lambda_\mu)\otimes_{\Lambda_\mu}\K=0\textrm{~for~} i=3,4$.
So, 
\[ \rank_{\Lambda_\mu} H_2(W,X_{L_i};\Lambda_\mu)=\chi(C_*)=\chi(C_*(W,X_{L_i};\Z))=\rank_\Z H_2(W,X_{L_i};\Z)=2r \] for $i=0,1$. As in  Lemma \ref{abelianlemma}, the existence of $1$-lagrangians and $1$-duals implies that the following is exact for $i=0,1$:
\[ tH_2(W,X_{L_i};\Lambda_\mu)\to H_1(X_{L_i};\Lambda_\mu)\to H_1(W;\Lambda_\mu)\to tH_1(W,X_{L_i};\Lambda_\mu).\]
(Note that $H_1(W,X_{L_i};\Lambda_\mu)=tH_1(W,X_{L_i};\Lambda_\mu)$ for $i=0,1$.)
In particular, (1) is proved because \[ \rank_{\Lambda_\mu}H_1(X_{L_0};\Lambda_\mu)=\rank_{\Lambda_\mu}H_1(W;\Lambda_\mu)=\rank_{\Lambda_\mu}H_1(X_{L_1};\Lambda_\mu).\]

The following is aslo exact for $i=0,1$:
\[ tH_2(W, X_{L_i};\Lambda_\mu)\to tH_1(X_{L_i};\Lambda_\mu)\to tH_1(W;\Lambda_\mu)\to tH_1(W,X_{L_i};\Lambda_\mu).\]
Now, fix $i$. Denote the 0-th Alexander polynomial of these modules and $tH_1(\partial W;\Lambda_\mu)$ by \[\Delta_2^{\vphantom{T}},\Delta_{L_i}^T,\Delta_W^{\vphantom{T}},\Delta_1^{\vphantom{T}},\textrm{~and~}\Delta_{\partial W}^{\vphantom{T}}\] respectively. (Of course, $\Delta_2$ and $\Delta_1$ depend on $i$.) 

Note that $H_1(W)\cong \Z^\mu$ and $H_2(W,X_{L_i};\Z)\cong\Z^{2r}$, so $\rank_{\Lambda_\mu} H_2(W,X_{L_i};\Lambda_\mu)=2r$. Applying Lemma \ref{lemma2.1kawauchi} to $(X,X_0)=(W,X_{L_i})$ for the cases $(k,l)=(2,2r)$ and $(1,0)$, we obtain $|\Delta_2(1,\ldots,1)|=|\Delta_1(1,\ldots,1)|=1$. Using Remark \ref{easylemma} (2), $\Delta_{L_i}^Tg_i^{\vphantom{'}}\overset{\cdot}=\Delta_W^{\vphantom{T}}g_i'$ for some $g_i^{\vphantom{'}},g_i'\in \Lambda_\mu$ with $|g_i^{\vphantom{'}}(1,\ldots,1)|=|g_i'(1,\ldots,1)|=1$ for $i=0,1$. In particular, $\Delta_{L_0}^Tg\overset{\cdot}=\Delta_{L_1}^Tg'$ for some $g,g'\in \Lambda_\mu$ with $|g(1,\ldots,1)|=|g'(1,\ldots,1)|=1$. 

Since $\Lambda_\mu$ is a unique factorization domain, we can split $\Delta_{L_i}^T=u_iv_i$ and $\Delta_{\partial W}^{\vphantom{T}}=uv$ uniquely (up to units of $\Lambda_\mu$) so that $v_0,v_1,v$ consist of all irreducible factors $f\in \Lambda_\mu$ with $|f(1,\ldots,1)|\neq 1$ in $\Delta_{L_0}^T,\Delta_{L_1}^T,\Delta_{\partial W}^{\vphantom{T}}$. From $\Delta_{L_0}^Tg\overset{\cdot}=\Delta_{L_1}^Tg'$ and $|g(1,\ldots,1)|=|g'(1,\ldots,1)|=1$, $v_0\overset{\cdot}=v_1$.

From the Mayer-Vietoris sequence, the following is exact: 
\[ tH_1(\partial X_{L_0};\Lambda_\mu)\to tH_1(X_{L_0};\Lambda_\mu)\oplus tH_1(X_{L_1};\Lambda_\mu)\to tH_1(\partial W;\Lambda_\mu)\to tH_0(\partial X_{L_0};\Lambda_\mu). \] 
The extreme terms are $\prod\limits_{i=1}^\mu(t_i-1)$-torsion. It follows that $\Delta_{\partial W}^{\vphantom{T}}\lambda\overset{\cdot}=\Delta_{L_0}^T\Delta_{L_1}^T\lambda'$ for some factors $\lambda,\lambda'$ of $\prod\limits_{i=1}^\mu(t_i-1)$. By the reciprocity of Blanchfield pairing \cite{Blanchfield:1957-1}, $\Delta_{L_i}^T\overset{\cdot}= \overline{\Delta_{L_i}^T}$ for $i=0,1$. Now, we have 
\[ u\overset{\cdot}= u_0u_1\overset{\cdot}= \overline{u_0}u_1.\]
By Theorem B, we proved that the Blanchfield form of $\partial W$ is neutral, which implies that $\Delta_{\partial W}=h\overline{h}$ for some $h\in\Lambda_\mu$ by \cite[Theorem 3.27]{Hillman:2012-1}. In particular, 
\[u\overset{\cdot}= f\overline{f}\textrm{~for some~}f\in \Lambda_\mu\textrm{~with~} |f(1,\ldots,1)|=1.\]
Combining all these observations,  
\[ \Delta_{L_0}^Tf\overline{f}\overset{\cdot}=u_0v_0u=u_0 \overline{u_0}u_1v_0\overset{\cdot}=u_0\overline{u_0}u_1v_1\overset{\cdot}=\Delta_{L_1}^Tu_0\overline{u_0}.\]
Here, $f$ and $u_0$ satisfy the conditions $|f(1,\ldots,1)|=1$, $|u_0(1,\ldots,1)|=1$. This completes the proof.
\end{proof}
\begin{remark0}It should be noted that Theorem C is not a direct consequence of Theorem B. From Theorem B without Lemma \ref{lemma2.1kawauchi}, one may deduce that if $X_{L_0}$ and $X_{L_1}$ are 1-solvable cobordant, then 
\begin{enumerate}
\item $\rank_{\Lambda_\mu}H_1(X_{L_0};\Lambda_\mu)=\rank_{\Lambda_\mu} H_1(X_{L_1};\Lambda_\mu)$ and
\item $\Delta_{L_0}^Tf_0\overline{f_0}\overset{\cdot}=\Delta_{L_1}^Tf_1\overline{f_1}$ 
\end{enumerate}
for some $f_0,f_1\in \Lambda_{\mu S}-\{0\}$. Lemma \ref{lemma2.1kawauchi} is crucial to obtain the stronger conclusion that we can choose $f_0,f_1\in \Lambda_\mu$ such that $|f_0(1,\ldots,1)|=|f_1(1,\ldots,1)|=1$.
\end{remark0}
Finally, we mention what can be deduced from Theorems B and C for the special case of 2-component links with linking number 1. Note that by the work of Levine \cite{Levine:1982-1}, the Blanchfield form (without localization) $b_L\colon tH_1(X_L;\Lambda_2)\times tH_1(X_L;\Lambda_2)\to \K/\Lambda_2$ is non-singular.
\begin{corollaryD}Suppose $L$ is a 2-component link with linking number 1. If $X_L$ and $X_H$ are 1-solvable cobordant, then 
\begin{enumerate}
\item $[b_L]=0\in W(\K,\Lambda_2,-)$,
\item $\beta(L)=0$,
\item $\Delta_0(L)\overset{\cdot}=f\overline{f}$ 
for some $f\in \Lambda_2$ such that $|f(1,1)|=1$. 
\end{enumerate} 
In particular, the conclusion holds if $L$ and $H$ are height 3 Whitney tower/grope concordant.

\end{corollaryD}
\begin{proof}Let $L$ be a 2-component link with linking number 1. Assume that $X_L$ and $X_H$ are 1-solvable cobordant. Since $X_H=S^1\times S^1\times I$ and the $\Z\oplus\Z$ cover of $X_H$ is $\R\times \R\times I$, $[b_H]=0,\beta(H)=0$ and $\Delta_0(H)=1$. This shows (1) and (2). With the notation in the proof of Theorem C (applied to $L_0=H$ and $L_1=L$), $u_0=1$ and 
\[ \Delta_0(H)f\overline{f}\overset{\cdot}=\Delta_0(L)u_0\overline{u_0}\]
for some $f\in \Lambda_2$ such that $|f(1,1)|=1$.
Since $\Delta_0(H)=1$ and $u_0=1$, $\Delta_0(L)=f\overline{f}$ for some $f\in \Lambda_2$ such that $|f(1,1)|=1$. This completes the proof of (3).
\end{proof} 
%Also, the following is an interesting open question : \emph{Is there a difference between $0.5$-solvable cobordism and $1$-solvable cobordism ?}

\bibliographystyle{amsalpha}
\renewcommand{\MR}[1]{}
\bibliography{research}

\def\cprime{$'$}
\providecommand{\bysame}{\leavevmode\hbox to3em{\hrulefill}\thinspace}
\providecommand{\MR}{\relax\ifhmode\unskip\space\fi MR }
% \MRhref is called by the amsart/book/proc definition of \MR.
\providecommand{\MRhref}[2]{%
  \href{http://www.ams.org/mathscinet-getitem?mr=#1}{#2}
}
\providecommand{\href}[2]{#2}
\begin{thebibliography}{COT04}

\bibitem[Bla57]{Blanchfield:1957-1}
R.~C. Blanchfield, \emph{Intersection theory of manifolds with operators with
  applications to knot theory}, Ann. of Math. (2) \textbf{65} (1957), 340--356.
  \MR{19,53a}

\bibitem[CF77]{Crowell-Fox:1977-1}
R.~H. Crowell and R.~H. Fox, \emph{Introduction to knot theory},
  Springer-Verlag, New York, 1977, Reprint of the 1963 original, Graduate Texts
  in Mathematics, No. 57. \MR{56 \#3829}

\bibitem[CG78]{Casson-Gordon:1978-1}
A.~Casson and C.~McA. Gordon, \emph{On slice knots in dimension three},
  Algebraic and geometric topology (Proc. Sympos. Pure Math., Stanford Univ.,
  Stanford, Calif., 1976), Part 2, Amer. Math. Soc., Providence, R.I., 1978,
  pp.~39--53. \MR{81g:57003}

\bibitem[CG86]{Casson-Gordon:1986-1}
\bysame, \emph{Cobordism of classical knots}, \`A la recherche de la topologie
  perdue, Birkh\"auser Boston, Boston, MA, 1986, With an appendix by P. M.
  Gilmer, pp.~181--199. \MR{900 252}

\bibitem[Cha07]{Cha:2007-1}
J.~C. Cha, \emph{The structure of the rational concordance group of knots},
  Mem. Amer. Math. Soc. \textbf{189} (2007), no.~885, x+95. \MR{MR2343079}

\bibitem[Cha09]{Cha:2009-1}
\bysame, \emph{Structure of the string link concordance group and
  {H}irzebruch-type invariants}, Indiana Univ. Math. J. \textbf{58} (2009),
  no.~2, 891--927. \MR{MR2514393}

\bibitem[Cha10]{Cha:2010-1}
\bysame, \emph{Link concordance, homology cobordism, and {H}irzebruch-type
  defects from iterated $p$-covers}, J. Eur. Math. Soc. (JEMS) \textbf{12}
  (2010), no.~3, 555--610.

\bibitem[Cha13]{Cha:2010-01}
\bysame, \emph{Amenable ${L}^2$-theoretic methods and knot concordance}, Int.
  Math. Res. Not. IMRN (2013), no.~15, 1--36.

\bibitem[Cha14]{Cha:2012-1}
\bysame, \emph{Symmetric {W}hitney tower cobordism for bordered 3-manifolds and
  links}, Trans. Amer. Math. Soc. \textbf{366} (2014), no.~6, 3241--3273.

\bibitem[CHL08]{Cochran-Harvey-Leidy:2008-1}
T.~D. Cochran, S.~Harvey, and C.~Leidy, \emph{Link concordance and generalized
  doubling operators}, Algebr. Geom. Topol. \textbf{8} (2008), no.~3,
  1593--1646. \MR{MR2443256 (2009h:57014)}

\bibitem[CHL09]{Cochran-Harvey-Leidy:2009-01}
\bysame, \emph{Knot concordance and higher-order {B}lanchfield duality}, Geom.
  Topol. \textbf{13} (2009), no.~3, 1419--1482. \MR{MR2496049}

\bibitem[CHL11]{Cochran-Harvey-Leidy:2009-02}
\bysame, \emph{Primary decomposition and the fractal nature of knot
  concordance}, Math. Ann. \textbf{351} (2011), no.~2, 443--508.

\bibitem[Cim04]{Cimasoni:2004-1}
D.~Cimasoni, \emph{A geometric construction of the conway potential function},
  Comment. Math. Helv. \textbf{79} (2004), no.~1, 124--146.

\bibitem[CK08]{Cochran-Kim:2004-1}
T.~D. Cochran and T.~H. Kim, \emph{Higher-order {A}lexander invariants and
  filtrations of the knot concordance group}, Trans. Amer. Math. Soc.
  \textbf{360} (2008), no.~3, 1407--1441 (electronic). \MR{MR2357701
  (2008m:57008)}

\bibitem[Coo82]{Cooper:1982-1}
D.~Cooper, \emph{The universal abelian cover of a link}, Low-dimensional
  topology (Bangor, 1979), Cambridge Univ. Press, Cambridge, 1982, pp.~51--66.
  \MR{83g:57002}

\bibitem[COT03]{Cochran-Orr-Teichner:2003-1}
T.~D. Cochran, K.~E. Orr, and P.~Teichner, \emph{Knot concordance, {W}hitney
  towers and {$L\sp 2$}-signatures}, Ann. of Math. (2) \textbf{157} (2003),
  no.~2, 433--519. \MR{1 973 052}

\bibitem[COT04]{Cochran-Orr-Teichner:2004-1}
\bysame, \emph{Structure in the classical knot concordance group}, Comment.
  Math. Helv. \textbf{79} (2004), no.~1, 105--123. \MR{MR2031301 (2004k:57005)}

\bibitem[CP]{Cha-Powell:2013-1}
J.~C. Cha and M.~Powell, \emph{Non-concordant links with homology cobordant
  zero framed surgery manifolds}, arXiv:1309.0926, to appear in Pacific J.
  Math.

\bibitem[CST11]{Conant-Schneiderman-Teichner:2011-1}
J.~Conant, R.~Schneiderman, and P.~Teichner, \emph{Higher-order intersections
  in low-dimensional topology}, Proc. Natl. Acad. Sci. USA \textbf{108} (2011),
  no.~20, 8131--8138. \MR{2806650}

\bibitem[CT07]{Cochran-Teichner:2003-1}
T.~D. Cochran and P.~Teichner, \emph{Knot concordance and von {N}eumann
  {$\rho$}-invariants}, Duke Math. J. \textbf{137} (2007), no.~2, 337--379.
  \MR{MR2309149 (2008f:57005)}

\bibitem[FP12]{Friedl-Powell:2010-1}
S.~Friedl and M.~Powell, \emph{An injectivity theorem for {C}asson-{G}ordon
  type representations relating to the concordance of knots and links}, Bull.
  Korean Math. Soc. \textbf{49} (2012), no.~2, 395--409.

\bibitem[FP14]{Friedl-Powell:2011-1}
\bysame, \emph{Links not concordant to the {H}opf link}, Math. Proc. Cambridge
  Philos. Soc. \textbf{156} (2014), no.~3, 425--459.

\bibitem[FQ90]{Freedman-Quinn:1990-1}
M.~H. Freedman and F.~Quinn, \emph{Topology of 4-manifolds}, Princeton
  Mathematical Series, vol.~39, Princeton University Press, Princeton, NJ,
  1990. \MR{MR1201584 (94b:57021)}

\bibitem[Har08]{Harvey:2006-1}
S.~Harvey, \emph{Homology cobordism invariants and the
  {C}ochran-{O}rr-{T}eichner filtration of the link concordance group}, Geom.
  Topol. \textbf{12} (2008), 387--430.

\bibitem[Hat02]{Hatcher:2002}
A~Hatcher, \emph{Algebraic topology}, Cambridge Univ. Press, Cambridge, 2002.

\bibitem[Hil12]{Hillman:2012-1}
J.~Hillman, \emph{Algebraic invariants of links}, second edition ed., Series on
  Knots and Everything, vol.~52, World Scientific Publishing Co. Pte. Ltd.,
  Hackensack, NJ, 2012. \MR{2003k:57014}

\bibitem[Hor10]{Horn:2010-1}
P.~D. Horn, \emph{The non-triviality of the grope filtrations of the knot and
  link concordance groups}, Comment. Math. Helv. \textbf{85} (2010), no.~4,
  751--773. \MR{2718138}

\bibitem[Kaw78]{Kawauchi:1978-1}
A.~Kawauchi, \emph{On the {A}lexander polynomials of cobordant links}, Osaka J.
  Math. \textbf{15} (1978), no.~1, 151--159. \MR{58 \#7599}

\bibitem[Lev69]{Levine:1969-1}
J.~P. Levine, \emph{Knot cobordism groups in codimension two}, Comment. Math.
  Helv. \textbf{44} (1969), 229--244. \MR{39 #7618}

\bibitem[Lev82]{Levine:1982-1}
\bysame, \emph{The module of a {$2$}-component link}, Comment. Math. Helv.
  \textbf{57} (1982), no.~3, 377--399. \MR{84h:57003}

\bibitem[Lev94]{Levine:1994-1}
\bysame, \emph{Link invariants via the eta invariant}, Comment. Math. Helv.
  \textbf{69} (1994), no.~1, 82--119. \MR{95a:57009}

\bibitem[Sch06]{Schneiderman:2006-1}
R.~Schneiderman, \emph{Whitney towers and gropes in 4-manifolds}, Trans. Amer.
  Math. Soc. \textbf{358} (2006), no.~10, 4251--4278 (electronic). \MR{2231378
  (2007e:57018)}

\end{thebibliography}
\end{document}